\documentclass[11pt, reqno]{amsart}
\usepackage[T1]{fontenc}
\usepackage{amsmath}
\usepackage{amssymb}
\usepackage{amsthm}
\usepackage[left=3.5cm, right=3.5cm, paperheight=11.9in]{geometry}
\usepackage{hyperref}
\usepackage{fancyhdr}
\usepackage{enumitem}
\usepackage{nicefrac}
\usepackage{bm}
\usepackage{mathrsfs}
\usepackage{graphicx}
\usepackage[utf8]{inputenc}
\usepackage{cancel}
\usepackage{mathtools}

\newtheorem{thm}{Theorem}[section]
\newtheorem{cor}[thm]{Corollary}
\newtheorem{lem}[thm]{Lemma}

\newtheorem{claim}[thm]{Claim}
\newtheorem{prop}[thm]{Proposition}
\newtheorem{probl}[thm]{Problem}

\theoremstyle{definition} 

\newtheorem{rmk}[thm]{Remark}
\newtheorem{example}[thm]{Example}
\newtheorem{nota}[thm]{Notation}
\numberwithin{equation}{section}

\def \suppt {\operatorname{suppt}}
\newcommand{\Q}{\mathbb{Q}}
\newcommand{\R}{\mathbb{R}}
\newcommand{\e}{\varepsilon}
\DeclareMathOperator{\dens}{dens}
\renewcommand{\leq}{\leqslant}
\renewcommand{\geq}{\geqslant}
\newcommand{\Span}{\operatorname{span}}
\newcommand{\closedSpan}{\overline{\operatorname{span}}}

\newcommand{\cut}{\mathord{\upharpoonright}}
\newcommand{\w}{\mathrm{w}}
\newcommand{\interior}{\operatorname{int}}

\newcounter{smallromans}

\newenvironment{romanenumerate}
{\begin{list}{{\normalfont\textrm{(\roman{smallromans})}}}
  {\usecounter{smallromans}\setlength{\itemindent}{0cm}
   \setlength{\leftmargin}{5.5ex}\setlength{\labelwidth}{5.5ex}
   \setlength{\topsep}{.5ex}\setlength{\partopsep}{.5ex}
   \setlength{\itemsep}{0.1ex}}}
{\end{list}}

\AtBeginDocument{
   \def\MR#1{}
}

\hypersetup{
    pdftitle={},
    pdfauthor={Paolo Leonetti, Tommaso Russo, and Jacopo Somaglia},
    pdfmenubar=false,
    pdffitwindow=true,
    pdfstartview=FitH,
    colorlinks=true,
    linkcolor=blue,
    citecolor=green,
    urlcolor=cyan
}

\providecommand{\MR}[1]{}

\providecommand{\MR}{\relax\ifhmode\unskip\space\fi MR }

\providecommand{\href}[2]{#2}

\usepackage{todonotes}

\begin{document}
\title{Dense lineability and spaceability\\ in certain subsets of $\ell_{\infty}$}

\author[P.~Leonetti]{Paolo Leonetti}
\address[P.~Leonetti]{Department of Economics, Universit\`a degli Studi dell'Insubria, via Monte Generoso 71, Varese 21100, Italy}
\email{leonetti.paolo@gmail.com}

\author[T.~Russo]{Tommaso Russo}
\address[T.~Russo]{Universit\"{a}t Innsbruck, Department of Mathematics, Technikerstra\ss e 13, 6020 Innsbruck, Austria; and Department of Mathematics, Faculty of Electrical Engineering, Czech Technical University in Prague, Technick\'a 2, 166 27 Prague 6, Czech Republic}
\email{tommaso.russo@uibk.ac.at, tommaso.russo.math@gmail.com}

\author[J.~Somaglia]{Jacopo Somaglia}
\address[J.~Somaglia]{Politecnico di Milano, Dipartimento di Matematica, Piazza Leonardo da Vinci 32, 20133 Milano, Italy}
\email{jacopo.somaglia@polimi.it}

\thanks{P.~Leonetti is grateful to PRIN 2017 (grant 2017CY2NCA) for financial support. 
T.~Russo and J.~Somaglia were supported by Gruppo Nazionale per l'Analisi Matematica, la Probabilit\`a e le loro Applicazioni (GNAMPA) of Istituto Nazionale di Alta Matematica (INdAM), Italy.}


\begin{abstract} We investigate dense lineability and spaceability of subsets of $\ell_\infty$ with a prescribed number of accumulation points. We prove that the set of all bounded sequences with exactly countably many accumulation points is densely lineable in $\ell_\infty$, thus complementing a recent result of Papathanasiou who proved the same for the sequences with continuum many accumulation points. We also prove that these sets are spaceable. We then consider the same problems for the set of bounded non-convergent sequences with a finite number of accumulation points. We prove that such a set is densely lineable in $\ell_\infty$ and that it is nevertheless not spaceable. The said problems are also studied in the setting of ideal convergence and in the space $\mathbb{R}^\omega$.
\end{abstract}
\maketitle
\thispagestyle{empty}

\section{Introduction}\label{sec:intro}
A subset $M$ of a vector space $X$ is said to be \emph{lineable} (resp.~\emph{$\kappa$-lineable}, for a cardinal $\kappa$) if $M\cup\{0\}$ contains a vector space of infinite dimension (resp.~of dimension $\kappa$). Lineability problems have been investigated in several areas of Mathematical Analysis; we refer to, \emph{e.g.},  \cite{MR3445906, AvilTodo, MR3119823, FGK, FSTZ, FZ, PZpoly, Rmoutil} for a rather non-exhaustive list of results. Let us just quote here the seminal result of Gurariy \cite{Gurariy} that the set of continuous, nowhere differentiable functions is lineable in $C([0,1])$. There are several variants and strengthenings of the above definition. If $X$ is a Banach space (or, more generally, a topological vector space), a subset $M$ of $X$ is \emph{spaceable} if $M\cup\{0\}$ contains a closed infinite-dimensional subspace; $M$ is \emph{densely lineable} in $X$ if $M\cup\{0\}$ contains a linear subspace that is dense in $X$.

A particular case where these properties have been considered in the literature is when the subset $M$ has the form $X\setminus Y$, where $Y$ is a closed subspace of $X$; in which setting there are simple and complete results, see \cite{BernalCabreraJFA, KitTim, Wilansky1975}. In particular, $X\setminus Y$ is spaceable if and only if $X\setminus Y$ is lineable, if and only if $Y$ has infinite codimension (\emph{i.e.}, $X/Y$ is infinite-dimensional) \cite{Wilansky1975}. Moreover, for separable $X$, these conditions are equivalent to $X\setminus Y$ being densely lineable in $X$ \cite{BernalCabreraJFA}. For non-separable spaces, Papathanasiou \cite{Papa2021} very recently proved that $\ell_\infty \setminus c_0$ is densely lineable in $\ell_\infty$. It is however most unfortunate that his result is actually consequence of \cite{BernalCabreraJFA}; indeed, the very same proof of \cite[Theorem 2.5]{BernalCabreraJFA} gives the complete characterisation that $X\setminus Y$ is densely lineable in $X$ if and only if $\dim(X/Y)\geq \dens(X)$. For the sake of completeness, we record this result in Corollary \ref{cor: dense lin iff}. \smallskip

Yet, inspection of the proof in \cite{Papa2021} gives the following more precise result: there is a dense subspace $V$ of $\ell_\infty$ such that every non-zero vector in $V$ has exactly continuum many accumulation points. This result was the starting point of our research, as we were pondering lineability results for subsets of $\ell_\infty$ with a prescribed number of accumulation points (see \cite{MR3003676} for some results in a similar direction). Before we can explain our results, it will be convenient to introduce a piece of notation that we shall use extensively throughout the paper. For a vector $x\in\ell_\infty$, we indicate by $\mathrm{L}_x$ the set of its accumulation points. If $\kappa$ is a cardinal number, $\mathrm{L}(\kappa)$ stands for the set of all $x\in\ell_\infty$ that have exactly $\kappa$ accumulation points; in other words, 
\begin{equation*}
    \mathrm{L}(\kappa)= \{x\in \ell_\infty\colon |\mathrm{L}_x|=\kappa\}.    
\end{equation*}
In this notation, the result in \cite{Papa2021} asserts that $\mathrm{L}(\mathfrak{c})$ is densely lineable in $\ell_\infty$. As it turns out, this more precise version can also be easily derived from \cite[Theorem 2.5]{BernalCabreraJFA}, since we can write $\mathrm{L}(\mathfrak{c})=\ell_\infty \setminus Y$ where $Y$ is the linear subspace $\bigcup_{\kappa \leq \omega} \mathrm{L}(\kappa)$ (see Remark \ref{rmk: Lc dense lin}). Similarly, we also show that $\mathrm{L}(\omega)$ is densely lineable in $\ell_\infty$ (Theorem \ref{th: L omega dense lin}). Notice that $\mathrm{L}(\kappa)=\emptyset$ for uncountable $\kappa<\mathfrak{c}$, as $\mathrm{L}_x$ is a closed set; hence, these results settle the situation for sequences with infinitely many accumulation points. Next, in Theorem \ref{thm:finitepart} we prove that the set $\bigcup_{2\leq n<\omega}\mathrm{L}(n)$ (that is, the set of non-convergent sequences with finitely many accumulation points) is also densely lineable in $\ell_\infty$. We have to exclude $n=1$ in the above union, since $\bigcup_{1\leq n<\omega}\mathrm{L}(n)$ clearly is a dense subspace of $\ell_\infty$.

Having answered the problem for what concerns dense lineability, in Section \ref{sec: spaceability} we turn our attention to spaceability of the said sets. Here, the results cannot be derived from the characterisation mentioned in the second paragraph, since the result in \cite[\S 6]{Wilansky1975} only works when $Y$ is a closed subspace of $X$. This assumption is not available in our setting since the linear subspaces that we consider are $\bigcup_{\kappa \leq \omega} \mathrm{L}(\kappa)$ and $\bigcup_{\kappa< \omega} \mathrm{L}(\kappa)$ that are both dense in $\ell_\infty$. Yet, we give a simple direct proof that $\mathrm{L}(\mathfrak{c})$ and $\mathrm{L}(\omega)$ are spaceable (Theorem \ref{thm:Lomegaspaceable}). On the other hand, the main result of the section is of negative nature as it asserts that $\bigcup_{2\leq n<\omega}\mathrm{L}(n)$ is not spaceable (Theorem \ref{thm: finite not spaceble}).

The proof of the latter relies on a result which we consider to be of independent interest: if $A\subseteq \{2,3,\dots\}$ is a non-empty finite interval, then $\bigcup_{n \in A}\mathrm{L}(n)$ (that is, the set of bounded sequences with a number of accumulations points prescribed by $A$) is $|A|$-lineable and, in addition, the lineability constant $|A|$ is sharp (Theorem \ref{thm: interval not lin}). This opens the way, in Section \ref{sec: finer}, to the search of several finer lineability results, in which we show that the lineability of $\bigcup_{n \in A}\mathrm{L}(n)$ is a much harder task when $A$ is not an interval. To wit, we prove that if $A\subseteq \{2,3,\dots\}$ is a sufficiently `sparse' infinite set, then $\bigcup_{n \in A}\mathrm{L}(n)$ is not even $2$-lineable; for example, the sets $\bigcup_{2\leq n<\omega}\mathrm{L}(n!)$ and $\bigcup_{1\leq n<\omega}\mathrm{L}(3^n)$ are not $2$-lineable (Corollary \ref{cor:infinitesetnot2lineable}). On the other hand, it is also possible that an infinite set $A$ contains no non-trivial intervals and yet $\bigcup_{n \in A}\mathrm{L}(n)$ is $\mathfrak{c}$-lineable. Indeed, we prove in Theorem \ref{th: odd c lineable} that the set $\bigcup_{1\leq n<\omega}\mathrm{L}(2n+1)$ is $\mathfrak{c}$-lineable. Finally, in Section \ref{sec: problems} we discuss extensions of our results when we replace convergent sequences and accumulation points with ideal convergent sequences and $\mathcal{I}$-cluster points respectively; we also discuss the same problems in the space $\R^\omega$ with the pointwise topology, instead of $\ell_\infty$. Finally, we collect some open problems that arise from our research.

\section{Preliminaries}\label{sec: prelim}
Our notation regarding Topology, Functional Analysis, and Set Theory is quite standard, as in most textbooks; we refer, \emph{e.g.}, to \cite{Eng, FHHMZ,  Jech} for unexplained notation and terminology. The unique caveat is that by \emph{subspace} of a normed space we understand a linear subspace, not necessarily closed. This will cause no confusion, since we will almost only consider subspaces that are either closed, or dense; when closedness is assumed, it will be stressed explicitly. For $x=(x(n))_{n\in\omega}\in\ell_{\infty}$ we define $\suppt(x)\coloneqq\{n\in\omega\colon x(n)\neq 0\}$. Given a set $\Gamma$, $|\Gamma|$ denotes the cardinality of $\Gamma$ and $\mathcal{P}(\Gamma)$ denotes the collection of all its subsets. We regard cardinal numbers as initial ordinal numbers; in particular, we write $\omega$ for the smallest infinite cardinal. The cardinality of continuum is denoted by $\mathfrak{c}$. When $A$ and $B$ are subsets of $\Gamma$, we write $A\subseteq ^* B$ to mean that $A\setminus B$ is finite; similarly, $A=^* B$ means that the symmetric difference between $A$ and $B$ is finite. $x\cut_A$ denotes the restriction of the function $x$ to the subset $A$ of its domain. For a subset $A\subseteq \Gamma$ we denote by $\bm{1}_{A}$ the characteristic function of $A$. A family $\mathscr{I}\subseteq \mathcal{P}(\omega)$ is \emph{independent} if for any distinct sets  $X_0,\dots, X_n, Y_0,\dots, Y_m \in \mathscr{I}$
\begin{equation*}
    X_0 \cap \dots \cap X_n \setminus (Y_0 \cup\dots\cup Y_m) \text{ is infinite.}
\end{equation*}
It is well known that $\omega$ contains an independent family of cardinality $\mathfrak{c}$ (see \cite[Lemma 7.7]{Jech}).

\smallskip

Recall that for a sequence $x\in \ell_\infty$ and $\eta\in \R$, $\eta$ is an \emph{accumulation point} of $x$ if $\{n \in \omega\colon |x(n)-\eta|<\e\}$ is infinite for all $\e>0$. Let us record explicitly the following notation that we mentioned already in the Introduction.
\begin{nota} For a vector $x\in\ell_\infty$ and a cardinal number $\kappa$ we write
\begin{gather*}
    \mathrm{L}_x\coloneqq \{\eta\in\R\colon \eta \text{ is an accumulation point of }x\} \\
    \mathrm{L}(\kappa)\coloneqq \{x\in \ell_\infty\colon |\mathrm{L}_x|=\kappa\}.
\end{gather*}
\end{nota}

Given $x,y\in\ell_\infty$ and $\alpha,\beta\in\R$ it is clear that $\mathrm{L}_{\alpha x + \beta y}\subseteq \{\alpha\xi+ \beta\eta \colon \xi\in\mathrm{L}_x, \eta\in\mathrm{L}_y\}$. For sequences with finitely many accumulation points we have the following simple consequence that we shall use several times.

\begin{lem}\label{lem:boundonnumberaccpoints} Let $x\in \mathrm{L}(k)$, $y \in \mathrm{L}(n)$ and $z\in\Span \{x,y\}$. Then $|\mathrm{L}_z|\leq kn$. Moreover, if $z=\alpha x+\beta y$ where both $\alpha$ and $\beta$ are different from $0$, then 
\begin{equation*}
    \max\left\{\frac{n}{k}, \frac{k}{n}\right\} \leq |\mathrm{L}_z| \leq kn.    
\end{equation*}
\end{lem}

\begin{proof} If $z=\alpha x+\beta y$, then $\mathrm{L}_z\subseteq \{\alpha\xi+ \beta\eta \colon \xi\in\mathrm{L}_x, \eta\in\mathrm{L}_y\}$ gives $|\mathrm{L}_z|\leq kn$. For the `Moreover' part, we can assume that $k\leq n$. As $\beta\neq 0$, $y\in\Span\{x,z\}$; hence the first part gives $n=|\mathrm{L}_y|\leq |\mathrm{L}_z|\cdot k$, and we are done.
\end{proof}

We conclude the section by giving a convenient representation for a sequence with finitely many accumulation points, that we shall use several times in what follows. We denote by $\sim_{c_0}$ the equivalence relation on $\ell_{\infty}$ defined by $$x \sim_{c_0} y \,\,\,\, \text{if and only if}\,\,\,\, x-y\in c_0.$$ 

\begin{lem}\label{lem: mod c0} 
Fix $n \in \omega$ and a sequence $x \in \mathrm{L}(n)$. Then there are a partition $\{S_1,\dots,S_n\}$ of $\omega$ in infinite sets and mutually distinct scalars $\xi_1,\dots,\xi_n$ such that
\begin{equation}\label{eq: mod c0}
    x \sim_{c_0} \xi_1 \bm{1}_{S_1} +\dots + \xi_n \bm{1}_{S_n}.
\end{equation}
Moreover, such a representation is unique up to the order and finite sets. More precisely, if $\eta_1\bm{1}_{T_1} +\dots+ \eta_m\bm{1}_{T_m}$ is another representation, then $n=m$ and there is a bijection $\sigma$ of $\{1,\dots,n\}$ such that $\eta_j= \xi_{\sigma(j)}$ and $T_j=^* S_{\sigma(j)}$, for every $j\in\{1,\dots,n\}$.
\end{lem}

\begin{rmk} Note that if $x$ admits a representation as in (\ref{eq: mod c0}), then $\mathrm{L}_x=\{\xi_1,\dots,\xi_n\}$ and $\|x\|\geq \max\{|\xi_i|\colon i\in\{1,\dots, n\}\}$. The shortest way to prove the second formula is to realise that $\max\{|\xi_i|\colon i\in\{1,\dots, n\}\}= \|q(x)\|_{\ell_\infty /c_0}\leq \|x\|$, where $q\colon \ell_\infty \to \ell_\infty /c_0$ is the quotient map.
\end{rmk}

\begin{proof}
Let $\{\xi_1,\dots,\xi_n\}$ be the accumulation points of $x$ and $\{S_1,\dots,S_n\}$ be a partition of $\omega$ in infinite sets such that $\lim_{k\in S_i}x(k)=\xi_i$ for every $i\in\{1,\dots,n\}$. Hence, we get $x\sim_{c_0} \xi_1 \bm{1}_{S_1} +\dots + \xi_n \bm{1}_{S_n}$. Conversely, if $x$ has the representation \eqref{eq: mod c0}, $\mathrm{L}_x= \{\xi_1,\dots,\xi_n\}$; therefore the scalars $\xi_1,\dots,\xi_n$ are uniquely determined up to the order. Suppose that there exists a second partition $\{T_1,\dots,T_n\}$ such that $x\sim_{c_0}\xi_1 \bm{1}_{T_1} +\dots + \xi_n \bm{1}_{T_n}$. Then, $\xi_1 \left(\bm{1}_{S_1}- \bm{1}_{T_1}\right) +\dots+ \xi_n \left(\bm{1}_{S_n}- \bm{1}_{T_n}\right) \in c_0$ and it attains finitely many values; hence such a sequence is eventually equal to zero, whence $S_i=^* T_i$ for every $i\in\{1,\dots,n\}$.
\end{proof}

\section{Dense lineability }\label{sec: finitely lims}
In this section we prove that $\mathrm{L}(\omega)$ and $\bigcup_{2\leq n<\omega}\mathrm{L}(n)$ are densely lineable in $\ell_\infty$, thus complementing the result from \cite{Papa2021} that $\mathrm{L}(\mathfrak{c})$ is densely lineable in $\ell_\infty$. As it turns out, both results are consequence of the extension of \cite[Theorem 2.5]{BernalCabreraJFA} that we mentioned already in the Introduction. Therefore, to begin with, we recall \cite[Theorem 2.5]{BernalCabreraJFA} in its general version. Even though the proof is essentially the same as in \cite{BernalCabreraJFA}, we provide a full argument for convenience of the reader. For a topological vector space $X$, $\dens(X)$ denotes the density character of $X$ and $\dim(X)$ its linear dimension (namely, the cardinality of an algebraic basis). If $Y$ is a linear subspace of $X$, the \emph{codimension} of $Y$ in $X$ is $\dim (X/Y)$. The weight of a topological space $X$ is denoted by $\w(X)$. 

\begin{lem}\label{lem: obv dense lin}
Let $X$ be a topological vector space and $Y$ be a linear subspace such that $\w(X)\leq \dim(X/Y)$. Then $X\setminus Y$ is densely lineable in $X$.
\end{lem}

\begin{proof} Let $\kappa\coloneqq \w(X)$ and $\{B_{\alpha}\}_{\alpha\in \kappa}$ be a topological basis for $X$. Assume that every $B_\alpha$ is non-empty. We build by transfinite induction vectors $\{x_{\alpha}\}_{\alpha\in \kappa}$ such that 
\begin{equation*}
    x_\alpha \in B_\alpha\setminus \Span(Y\cup \{x_\gamma\}_{\gamma\in\alpha}) \text{ for all } \alpha<\kappa.
\end{equation*}
Since $\interior(Y)=\emptyset$, there is $x_0\in B_0\setminus Y$. Let $\alpha<\kappa$ and suppose, by transfinite induction, that $x_\beta \in B_\beta\setminus \Span(Y\cup \{x_\gamma\}_{\gamma \in\beta})$ has been defined for every $\beta<\alpha$. Let $Y_{\alpha}\coloneqq \Span(Y\cup \{x_\beta\}_{\beta\in\alpha})$. The assumption that $Y$ has codimension at least $\kappa$ in $X$ gives $Y_\alpha\subsetneq X$, so $\interior(Y_{\alpha}) =\emptyset$. Hence, there is $x_\alpha \in B_{\alpha}\setminus Y_{\alpha}$. This shows the existence of the vectors $\{x_{\alpha}\} _{\alpha\in \kappa}$. The subset $\{x_{\alpha}\} _{\alpha\in \kappa}$ is dense in $X$, therefore $V\coloneqq \Span \{x_{\alpha}\}_{\alpha\in \kappa}$ is dense in $X$ and it is readily seen that $V\cap Y=\{0\}$.
\end{proof}

\begin{cor}\label{cor: dense lin iff} Let $X$ be a metrisable infinite-dimensional topological vector space with $\kappa=\dens(X)$ and $Y$ be a linear subspace. Then the following are equivalent:
\begin{romanenumerate}
    \item\label{cor item: dense lin} $X\setminus Y$ is densely lineable in $X$,
    \item\label{cor item: k lin} $X\setminus Y$ is $\kappa$-lineable,
    \item\label{cor item: dim X/Y} $\kappa\leq \dim(X/Y)$.
\end{romanenumerate}
\end{cor}
\begin{proof} Every metric space $X$ satisfies $\dens(X) = \w(X)$; hence, \eqref{cor item: dim X/Y} $\Rightarrow$ \eqref{cor item: dense lin} follows from Lemma \ref{lem: obv dense lin}. \eqref{cor item: dense lin} $\Rightarrow$ \eqref{cor item: k lin} is obvious. For \eqref{cor item: k lin} $\Rightarrow$ \eqref{cor item: dim X/Y}, take a subspace $V$ of $X$ with $\dim(V)=\kappa$ and such that $V\cap Y=\{0\}$; let $q\colon X\to X/Y$ be the canonical quotient map. Since $q\cut_V$ is injective, we have $\kappa= \dim(V)=\dim(q[V])\leq \dim(X/Y)$.
\end{proof}

In order to build a vector space of dimension $\mathfrak{c}$ inside $\mathrm{L}(\omega)$ we shall exploit the `strong' linear independence of geometric sequences in order to prevent non-trivial linear combinations to have only finitely many accumulation points. Similar uses of geometric sequences can be found in several places in the literature, \emph{e.g.}, \cite{CarSep14, HKRtams, HRjfa, Klee}. For this purpose, we will use the following standard lemma, see, \emph{e.g.}, \cite[Proposition 2.1]{CarSep14}; its proof is so simple that we give it here.

\begin{lem}\label{lem: geometric infinite values} 
Let $\lambda_0,\dots,\lambda_n\in(0,1)$ be mutually distinct scalars and let $\beta_0\dots,\beta_n \in\R$ not all equal to $0$. Then the sequence
\begin{equation*}
    \left( \beta_0\lambda_0^j + \dots + \beta_n\lambda_n^j \right)_{j\in\omega}
\end{equation*}
attains each of its values finitely many times. In particular, its range is an infinite set.
\end{lem}

\begin{proof} We can assume that $0<\lambda_0<\dots <\lambda_n<1$ and that $\beta_i\neq0$ for every $i \in \{0,\dots, n\}$. Moreover, the conclusion is clearly true when $n=0$, so we assume $n\geq 1$. Towards a contradiction, assume that there are a subsequence $(j_k)_{k\in\omega}$ of $\omega$ and $\gamma\in\R$ such that
\begin{equation*}
    \beta_0\lambda_0^{j_k} + \dots + \beta_n\lambda_n^{j_k} =\gamma \qquad \text{for every } k\in\omega.
\end{equation*}
Letting $k\to \infty$ shows that $\gamma=0$. Hence, we can divide by $\lambda_n^{j_k}$ to get
\begin{equation*}
    \beta_0\left(\frac{\lambda_0}{\lambda_n} \right)^{j_k} + \dots + \beta_{n-1} \left(\frac{\lambda_{n-1}}{\lambda_n}\right)^{j_k} =-\beta_n.
\end{equation*}
Since $\lambda_i<\lambda_n$ for $i\in \{0,\dots,n-1\}$, letting $k\to\infty$ gives $\beta_n=0$, a contradiction.
\end{proof}

\begin{thm}\label{th: L omega dense lin} $\mathrm{L}(\omega)$ is densely lineable in $\ell_\infty$.
\end{thm}

\begin{proof} Let $X\coloneqq \bigcup_{\kappa\leq \omega} \mathrm{L}(\kappa)$ and $Y\coloneqq \bigcup_{\kappa< \omega} \mathrm{L}(\kappa)$. Then $X$ and $Y$ are linear subspaces of $\ell_\infty$, $X$ is dense in $\ell_\infty$, and $X\setminus Y =\mathrm{L}(\omega)$. Therefore, if we prove that $\mathrm{L}(\omega)$ is $\mathfrak{c}$-lineable, Corollary \ref{cor: dense lin iff} would yield us that $\mathrm{L}(\omega)$ is densely lineable in $X$, hence also in $\ell_\infty$, which would conclude the proof.

To this aim, take disjoint subsets $(B_j)_{j\in\omega}$ of $\omega$ such that each $B_j$ is an infinite set. We can now define, for every $q\in(0,1)$, the following vector in $\ell_\infty$
\begin{equation}\label{eq:definitionFAkLomega}
    f_q\coloneqq \sum_{j=0}^\infty q^j \bm{1}_{B_j};
\end{equation}
it is sufficient to prove, as we now do, that no linear combination of $\{f_q\colon q\in (0,1)\}$ with non-zero scalars belongs to $Y$. For this aim, take mutually distinct $q_0 ,\dots, q_N\in (0,1)$ and non-zero scalars $d_0,\dots,d_N\in \R$. Then we can write
\begin{equation}\label{eq: x=y+z}
    x\coloneqq \sum_{n=0}^N d_n f_{q_n}= \sum_{j=0}^\infty \left( \sum_{n=0}^N d_n (q_n)^j \right) \bm{1}_{B_j}.
\end{equation}
Lemma \ref{lem: geometric infinite values} yields us that the sequence $(h_j)_{j \in \omega}$, defined by
\begin{equation}\label{eq: sequence with omega lims}
    h_j\coloneqq  \sum_{n=0}^N d_n (q_n)^j
\end{equation}
attains infinitely many distinct values. Since each value is attained on the corresponding infinite set $B_j$, it follows that the sequence $x$ admits infinitely many accumulation points. On the other hand, $h_j\to 0$; thus $\mathrm{L}_x$ is the countable set
\begin{equation*}
    \mathrm{L}_x= \left\{0, h_j \right\}_{j\in\omega}.
\end{equation*}
Hence, $x\in\mathrm{L}(\omega)$ and we are done.
\end{proof}

\begin{rmk}\label{rmk: Lc dense lin} A small variation of the above proof gives an alternative argument that $\mathrm{L}(\mathfrak{c})$ is densely lineable in $\ell_\infty$. Indeed, we now consider $X\coloneqq \ell_\infty$ and $Y\coloneqq \bigcup_{\kappa\leq \omega} \mathrm{L}(\kappa)$ and we only have to show that $X\setminus Y= \mathrm{L}(\mathfrak{c})$ is $\mathfrak{c}$-lineable. Next, for every $j\in\omega$ let $r_j\colon \omega\to (0,1)$ be a sequence such that $\suppt(r_j)=B_j$ and $\mathrm{L}_{r_j}=[0,1]$. Then replace the vectors $f_q$ given in \eqref{eq:definitionFAkLomega} with
\begin{equation*}
    f_q\coloneqq \sum_{j=0}^\infty q^j r_j \bm{1}_{B_j} \qquad(q\in(0,1)).
\end{equation*}

At this point, if $x$ is as in \eqref{eq: x=y+z} (with the extra factor $r_j$) and $h_j$ is as in \eqref{eq: sequence with omega lims}, take $j\in\omega$ with $h_j\neq0$. Then $x\cut_{B_j}=h_j r_j\cut_{B_j}\in \mathrm{L} (\mathfrak{c})$ (since $\mathrm{L}_{r_j}=[0,1]$). Thus, $x \in\mathrm{L} (\mathfrak{c})$, and we are done.
\end{rmk}

Finally, we cover the case of $\bigcup_{2\leq n<\omega}\mathrm{L}(n)$.
\begin{thm}\label{thm:finitepart} $\bigcup_{2\leq n<\omega}\mathrm{L}(n)$ is densely lineable in $\ell_\infty$.
\end{thm}

\begin{proof} In this case, we consider the linear subspaces of $\ell_\infty$ given by $X\coloneqq \bigcup_{n<\omega}\mathrm{L}(n)$ and $Y=\mathrm{L}(1)=c$ and, as above, we only need to prove that $\dim (X/Y)=\mathfrak{c}$. This is consequence of the fact that $X/c$ is dense in $\ell_\infty/c$, whose density character is $\mathfrak{c}$. Alternatively, one can take an independent family $\mathscr{I}\subseteq \mathcal{P} (\omega)$ of cardinality $\mathfrak{c}$; then it is easy to see that $\Span\{\bm{1}_A \colon A\in \mathscr{I}\}$ has dimension equal to $\mathfrak{c}$ and $\Span\{\bm{1}_A \colon A\in \mathscr{I}\}\cap \mathrm{L}(1) =\{0\}$.
\end{proof}

\section{Spaceability}\label{sec: spaceability}
In this section we focus on spaceability results for the sets $\bigcup_{2\leq n<\omega}\mathrm{L}(n)$, $\mathrm{L}(\omega)$, and $\mathrm{L}(\mathfrak{c})$. The main result is Theorem \ref{thm: finite not spaceble} asserting that $\bigcup_{2\leq n<\omega}\mathrm{L}(n)$ is not spaceable. A key ingredient in its proof is Theorem \ref{thm: interval not lin}, where we show that the subspace $\mathrm{L}(n)\cup \dots\cup \mathrm{L}(n+d)$ is $(d+1)$-lineable but not $(d+2)$-lineable. As a complement to this, we conclude the section with the easy result that $\mathrm{L}(\omega)$ and $\mathrm{L} (\mathfrak{c})$ are spaceable. 

The basic idea for the proof of Theorem \ref{thm: interval not lin} consists in finding certain linear combinations of vectors in a way to suitably increase or decrease the number of accumulation points. This will be achieved by means of the following lemmata. The first one will allow us to reduce the number of accumulation points as much as possible; the second asserts that small perturbations can't decrease the number of accumulation points; the last one claims that if no linear combination of two vectors increases the number of accumulation points, then the partitions associated to the vectors as in Lemma \ref{lem: mod c0} must be one finer than the other (modulo finite sets).
 
\begin{lem}\label{lem: decrease lims} 
Let $x_1,\dots,x_n\in \R^n$. Then there are scalars $c_1,\dots,c_n\in\R$, not all equal to zero, and $\gamma\in\R$ such that
\begin{equation*}
    c_1 x_1 +\dots+ c_n x_n =\gamma (1,\dots,1).
\end{equation*}
\end{lem}

\begin{proof} If the vectors $x_1,\dots,x_n$ are linearly independent, their linear span is $\R^n$, so there exists a linear combination that equals $(1,\dots,1)$. In the case they are linearly dependent, then there exists a non-trivial linear combination of them that gives $(0,\dots,0)$.
\end{proof}

\begin{lem}\label{lem: perturb many lims} Let $x\in\ell_\infty$ be a sequence with $|\mathrm{L}_x| <\infty$. There is $\e>0$ such that for all vectors $y\in\ell_\infty$ with $|\mathrm{L}_y|<\infty$ and $\|y\|<\e$,
\begin{equation*}
    |\mathrm{L}_{x+y}|\geq \max\{|\mathrm{L}_x|, |\mathrm{L}_y|\}.
\end{equation*}
\end{lem}

\begin{proof} Since $\mathrm{L}_x$ is a finite set, we may take $\e>0$ such that $\mathrm{L}_x$ is a $2\e$-separated set (\emph{i.e.}, $|\alpha-\beta|\geq2\e$ for distinct $\alpha,\beta\in \mathrm{L}_x$). Now take any $y\in\ell_\infty$ with $|\mathrm{L} _y|<\infty$ and $\|y\|<\e$. According to Lemma \ref{lem: mod c0}, we can write
\begin{equation*}
    x\sim_{c_0} \xi_1\bm{1}_{S_1} +\dots+ \xi_n\bm{1}_{S_n} \qquad\text{and}\qquad y\sim_{c_0} \eta_1\bm{1}_{T_1} +\dots+ \eta_k\bm{1}_{T_k}.
\end{equation*}
In order to check that $|\mathrm{L}_{x+y}|\geq |\mathrm{L}_y|=k$, fix $i\in\{1,\dots,k\}$ and take $j_i\in\{1,\dots,n\}$ such that $T_i\cap S_{j_i}$ is infinite. Therefore, $\eta_i + \xi_{j_i}$ is an accumulation point of $x+y$. Hence, if by contradiction $|\mathrm{L}_{x+y}| < |\mathrm{L}_y|=k$, there must be distinct indices $i,l\in\{1,\dots,k\}$ such that $\eta_i + \xi_{j_i} = \eta_l + \xi_{j_l}$. If $j_i=j_l$, we get the absurd that $\eta_i=\eta_l$. On the other hand, if $j_i\neq j_l$, then $2\e\leq |\xi_{j_i} -\xi_{j_l}|=|\eta_i - \eta_l|\leq 2\|y\|< 2\e$, a contradiction. The proof that $|\mathrm{L}_{x+y}|\geq |\mathrm{L}_x|$ is similar (starting with $i\in\{1,\dots,n\}$), therefore we omit it.
\end{proof}

\begin{lem}\label{lem: supporti inscatolati} Assume $x\in \mathrm{L}(n)$ and $y\in\mathrm{L}(k)$ have the representation
\begin{equation*}
    x\sim_{c_0} \xi_1\bm{1}_{S_1} +\dots+ \xi_n\bm{1}_{S_n} \qquad\text{and}\qquad y\sim_{c_0} \eta_1\bm{1}_{T_1} +\dots+ \eta_k\bm{1}_{T_k},
\end{equation*}
as in Lemma \ref{lem: mod c0}. Suppose also that $n\leq k$ and that every $z\in \Span\{x,y\}$ satisfies $|\mathrm{L}_z|\leq k$. Then for every $i\in\{1,\dots,k\}$, there exists $j\in\{1,\dots,n\}$ such that $T_i\subseteq^* S_j$. 
\end{lem}

\begin{proof} Suppose by contradiction that there is $i\in\{1,\dots,k\}$ such that $T_i \nsubseteq^* S_j$ for every $j\in\{1,\dots,n\}$. Then there are two distinct indices $j_1,j_2\in\{1,\dots,n\}$ such that $T_i\cap S_{j_1}$ and $T_i \cap S_{j_2}$ are both infinite. According to Lemma \ref{lem: perturb many lims}, for sufficiently small $\e>0$, $(x+\e y)\cut_{T_i}$ has at least two accumulation points (since $\xi_{j_1},\xi_{j_2}$ are accumulation points of $x\cut_{T_i}$) and $(x+\e y)\cut_{\omega \setminus T_i}$ has at least $k-1$ accumulation points ($y\cut_{\omega \setminus T_i}$ has $k-1$ accumulation points). Moreover, for small $\e>0$, the sets of accumulation points of the elements $(x+\e y)\cut_{T_i}$ and $(x+\e y)\cut_{\omega\setminus T_i}$ are disjoint. Thus, $x+\e y$ has at least $k+1$ accumulation points, and we are done.
\end{proof}

We are now ready for the first main result of the section.
\begin{thm}\label{thm: interval not lin} Let $n,d\in\omega$ be such that $n\geq 2$. Then $\mathrm{L}(n)\cup\dots\cup \mathrm{L}(n+d)$ is $(d+1)$-lineable, but not $(d+2)$-lineable.
\end{thm}

\begin{proof} We start by showing that $\mathrm{L}(n)\cup\dots\cup \mathrm{L}(n+d)$ is $(d+1)$-lineable. We claim that there are vectors $v_1,\dots,v_{n+d}\in \R^{d+1}$ such that, for all non-zero $\alpha\coloneqq (\alpha_0,\dots, \alpha_d)\in \R^{d+1}$ the set $\{\alpha\cdot v_j\}_{j=1}^{n+d}$ has cardinality at least $n$ ($\alpha\cdot v_j$ is the inner product of the vectors $\alpha$ and $v_j$ in $\R^{d+1}$). Since we didn't find a short proof of this claim, we decided to postpone its proof until Proposition \ref{prop: comb lin many values}. So, assuming the validity of the claim for now, take vectors $v_1,\dots,v_{n+d}\in \R^{d+1}$ as above and let $(B_j)_{j=1}^{n+d}$ be a partition of $\omega$ into infinite sets. For $k\in\{0,\dots, d\}$, define the vector
\begin{equation*}
    e_k\coloneqq \sum_{j=1}^{n+d} \bm{1}_{B_j} v_j(k)
\end{equation*}
and let $V\coloneqq \Span\{e_k\}_{k=0}^d$. Since $V\subseteq \Span\{\bm{1}_{B_1}, \dots, \bm{1}_{B_{n+d}}\}$ and the sets $B_j$ are disjoint and infinite, it follows that every vector in $V$ has at most $n+d$ accumulation points. Thus, we only need to prove that every non-zero vector in $V$ has at least $n$ accumulation points. Take scalars $\alpha_0,\dots, \alpha_d$, not all equal to $0$, and note that
\begin{equation*}
    \sum_{k=0}^d \alpha_k e_k= \sum_{j=1}^{n+d}  \left(\sum_{k=0}^d \alpha_k v_j(k)\right) \bm{1}_{B_j} = \sum_{j=1}^{n+d} \alpha\cdot v_j \bm{1}_{B_j}.
\end{equation*}
Once more, the fact that the sets $B_j$ are disjoint and infinite yields that the accumulation points of $\sum_{k=0}^d \alpha_k e_k$ are exactly
\begin{equation*}
    \{\alpha\cdot v_j \}_{j=1}^{n+d}.
\end{equation*}
By our assumption, such a set has cardinality at least $n$, as desired. \smallskip

Next, we shall show that $\mathrm{L}(n)\cup\dots\cup \mathrm{L}(n+d)$ is not $(d+2)$-lineable. Therefore, we fix $n\geq 2$ and $d\in\omega$ and assume, towards a contradiction, that $V$ is a vector space of dimension $d+2$ and $V\subseteq \mathrm{L}(n)\cup\dots \cup \mathrm{L}(n+d)\cup\{0\}$. Define $N\in\omega$ to be
\begin{equation*}
    N\coloneqq \max\{|\mathrm{L}_x|\colon x\in V\};
\end{equation*}
our assumption yields that $N\leq n+d$. Moreover, we can select $e_1\in V\cap \mathrm{L}(N)$; hence we can find a basis $\{e_1,\tilde{e}_2,\dots, \tilde{e}_{d+2}\}$ of $V$ that contains $e_1$. For $\e>0$ sufficiently small, the vectors $e_k\coloneqq \tilde{e}_k + \e e_1$ ($k=2,\dots, d+2$) belong to $\mathrm{L}(N)$: indeed, on the one hand, $|\mathrm{L}_{e_k}|\geq |\mathrm{L}_{e_1}|=N$ by Lemma \ref{lem: perturb many lims} and, on the other one, $|\mathrm{L}_{e_k}| \leq N$ by definition of $N$. Consequently, the set $\{e_1,\dots,e_{d+2}\}$ forms a basis of $V$ and each $e_k$ belongs to $\mathrm{L}(N)$.

Lemma \ref{lem: mod c0} allows us to write
\begin{equation*}
    e_1\sim_{c_0} \xi_1\bm{1}_{S_1} +\dots+ \xi_N\bm{1}_{S_N} \qquad\text{and}\qquad e_2\sim_{c_0} \eta_1\bm{1}_{T_1} +\dots+ \eta_N\bm{1}_{T_N}.
\end{equation*}
Since every vector in the linear span of $\{e_1,e_2\}$ has at most $N$ accumulation points, an appeal to Lemma \ref{lem: supporti inscatolati} assures us that for every $i\in\{1,\dots,N\}$ there is $j_i\in\{1,\dots,N\}$ such that $S_i\subseteq^* T_{j_i}$. $\{S_1,\dots,S_N\}$ being a partition, we conclude that indeed $S_i=^* T_{j_i}$. Up to a permutation in the representation of $e_2$, we can assume that $S_i=^* T_i$ for every $i\in\{1,\dots,N\}$. If we repeat the same argument with $e_1$ and $e_k$ for every $k\in\{3,\dots,d+2\}$, we obtain, for every $k\in\{1,\dots,d+2\}$ mutually distinct scalars $\xi_1(k),\dots, \xi_N(k)$ such that
\begin{equation*}
    e_k \sim_{c_0} \xi_1(k) \bm{1}_{S_1} +\dots+ \xi_N(k) \bm{1}_{S_N}.
\end{equation*}

Before we continue, let us observe that necessarily $N\geq d+2$. Indeed, if not, the vectors $\{\xi_1(k) \bm{1}_{S_1} +\dots+ \xi_N(k) \bm{1}_{S_N}\}_{k=1}^{d+2}$ would be linearly dependent, so there would exist scalars $\alpha_1,\dots,\alpha_{d+2}$, not all equal to zero and such that the corresponding linear combination of the vectors $\{\xi_1(k) \bm{1}_{S_1} +\dots+ \xi_N(k) \bm{1}_{S_N}\}_{k=1}^{d+2}$ would be equal to $0$. Hence $\alpha_1 e_1 +\dots+ \alpha_{d+2} e_{d+2}\in c_0$, a contradiction (note that the vector $\alpha_1 e_1 +\dots+ \alpha_{d+2} e_{d+2}$ cannot be equal to $0$, since the vectors $e_k$ are linearly independent by construction).

Since $N\geq d+2$, we can consider the vectors
\begin{equation*}
    \left(\xi_1(k),\dots, \xi_{d+2}(k)\right) \in \R^{d+2} \qquad (k\in\{1,\dots,d+2\})
\end{equation*}
and apply Lemma \ref{lem: decrease lims} to them. This yields us scalars $\alpha_1,\dots,\alpha_{d+2}\in\R$, not all equal to zero, and $\gamma \in\R$ such that
\begin{equation*}
    \sum_{k=1}^{d+2} \alpha_k\left(\xi_1(k),\dots, \xi_{d+2}(k)\right)= \gamma(1,\dots,1).
\end{equation*}
Consequently, we have
\begin{equation*}
    \sum_{k=1}^{d+2} \alpha_k e_k \sim_{c_0} \gamma \bm{1}_{S_1\cup\dots\cup S_{d+2}}+ \left(\sum_{k=1} ^{d+2} \alpha_k \xi_{d+3}(k)\right) \bm{1}_{S_{d+3}} +\dots+ \left(\sum_{k=1}^{d+2} \alpha_k \xi_N(k)\right) \bm{1}_{S_N}.
\end{equation*}
From this equation we conclude that the non-zero vector $\sum_{k=1}^{d+2} \alpha_k e_k$ (let us recall that $\{e_1,\dots,e_{d+2}\}$ is a linear basis of $V$) has at most $(N-d-1)$-accumulation points. Finally, recalling that $N\leq n+d$, we reach the contradiction that $\sum_{k=1}^{d+2} \alpha_k e_k$ has at most $(n-1)$-accumulation points.
\end{proof}

As a particular case we have the following result. Note that it directly yields that the set $\bigcup_{n\in A}\mathrm{L}(n)$ is never lineable, when $A$ is finite. Also, in case the set $A$ is not an interval, this corollary might fail to be sharp (see Corollary \ref{cor:LnLnplusdnot2lineable}).
\begin{cor}\label{cor: lims in A not lin} Let $A$ be a non-empty finite subset of $\omega$ such that $\min A\geq 2$. Then $\bigcup_{n\in A}\mathrm{L}(n)$ is not $(\mathrm{diam}(A)+2)$-lineable, where $\mathrm{diam}(A)\coloneqq \max A -\min A$. 
\end{cor}

We are now finally in position to pass to spaceability results. We first give the main result of the section concerning $\bigcup_{2\leq n<\omega}\mathrm{L}(n)$ and we then conclude the section with the simpler result for $\mathrm{L}(\omega)$ and $\mathrm{L}(\mathfrak{c})$.

\begin{thm}\label{thm: finite not spaceble}
$\bigcup_{2\leq n<\omega}\mathrm{L}(n)$ is not spaceable in $\ell_\infty$.
\end{thm}

\begin{proof} Towards a contradiction, assume that there is a closed, infinite-dimensional subspace $Y$ of $\ell_\infty$ such that $Y\subseteq \bigcup_{2\leq n<\omega} \mathrm{L}(n)\cup\{0\}$. According to Theorem \ref{thm: interval not lin}, $Y$ is contained in $\bigcup_{2\leq n\leq N} \mathrm{L}(n) \cup\{0\}$ for no $N\in\omega$, so $Y\cap \mathrm{L}(n)$ is non-empty for infinitely many $n\in\omega$. We shall build by induction a sequence $(\e_k)_{k\in\omega}$ of positive scalars with $\e_{k+1}\leq \frac{1}{2}\e_k$ for every $k\in\omega$, a sequence $(y_k)_{k\in\omega}$ of unit vectors in $Y$, and a strictly increasing sequence $(N_k)_{k\in\omega}$ of natural numbers, with the following properties (for every $k\in\omega$):
\begin{romanenumerate}
    \item $\e_0 y_0+\dots+ \e_k y_k \in \mathrm{L}(N_k)$,
    \item\label{item: many lims in perturbations} $\e_0 y_0+\dots+ \e_k y_k+ y \in \bigcup_{N_k\leq n<\omega}\mathrm{L}(n)$ for every $y\in Y$ with $\|y\| \leq 2 \e_{k+1}$. 
\end{romanenumerate}
Indeed, to start the induction, we set $\e_0\coloneqq  1$, we take any unit vector $y_0\in\ Y$ and we set $N_0 \coloneqq |\mathrm{L}_{y_0}|$. Assuming inductively to have already found $(\e_j)_{j\leq k}$, $(y_j)_{j\leq k}$, and $(N_j)_{j\leq k}$ as above, we apply Lemma \ref{lem: perturb many lims} to the vector $\e_0 y_0 +\dots+ \e_k y_k$ and we find $\e_{k+1}$ such that $\e_0 y_0+\dots+ \e_k y_k+y$ has at least $\max\{N_k,|\mathrm{L}_y|\}$ accumulation points for every $y\in Y$ with $\|y\|\leq 2\e_{k+1}$; clearly, we can also assume $2 \e_{k+1}\leq \e_k$. Since $Y\cap \mathrm{L}(n) \neq\emptyset$ for infinitely many $n\in\omega$, we are now in position to take a unit vector $y_{k+1}\in Y$ with $|\mathrm{L}_{y_{k+1}}|>N_k$. By Lemma \ref{lem: perturb many lims}, the cardinality of the accumulation points of $\e_0 y_0+\dots+ \e_{k+1} y_{k+1}$, which we denote $N_{k+1}$, is greater than $N_k$. This concludes the induction step.

Finally, since $Y$ is closed, $y\coloneqq \sum_{k=0}^\infty \e_k y_k\in Y$. However, for every $k\in\omega$ we have
\begin{equation*}
    \left\| \sum_{j=k+1}^\infty \e_j y_j\right\|\leq \sum_{j=k+1}^\infty \e_j\leq \sum_{j=0}^\infty 2^{-j}\e_{k+1} =2\e_{k+1}. 
\end{equation*}
Hence, if we write
\begin{equation*}
    y=\e_0 y_0+\dots+ \e_k y_k + \sum_{j=k+1} ^\infty \e_j y_j,
\end{equation*}
we see from (\ref{item: many lims in perturbations}) that $|\mathrm{L}_y|\geq N_k$. Since $k\in\omega$ was arbitrary and $N_k\to\infty$ as $k\to\infty$, we conclude that $y\notin \bigcup_{2 \leq n<\omega}\mathrm{L}(n)$, a contradiction.
\end{proof}

\begin{thm}\label{thm:Lomegaspaceable}
$\mathrm{L}(\omega)$ and $\mathrm{L}(\mathfrak{c})$ are spaceable in $\ell_\infty$. More precisely, $\mathrm{L}(\omega) \cup\{0\}$ contains $c_0$ isometrically and $\mathrm{L}(\mathfrak{c}) \cup\{0\}$ contains $\ell_\infty$ isometrically.
\end{thm}
\begin{proof} 
We first consider the case of $\mathrm{L}(\omega)$. Let $(A_{n,k})_{n,k \in \omega}$ be a partition of $\omega$ into infinite sets and define the vectors 
\begin{equation}\label{eq: en for spaceable}
e_n\coloneqq \sum_{k=0}^\infty a_k\cdot \bm{1}_{A_{n,k}},
\end{equation}
where $a_k=2^{-k}$ for $k\in\omega$. Each $e_n$ is a unit vector and $e_n\in \mathrm{L}(\omega)$ for each $n\in\omega$. Moreover, $\suppt(e_n)= \bigcup_{k\in\omega} A_{n,k}$, hence the vectors $e_n$ are disjointly supported. Thus the map $(\alpha_n)_{n\in\omega}\mapsto \sum_{n=0}^\infty \alpha_n e_n$ is an isometry from $c_0$ onto $Y\coloneqq \closedSpan\{e_n\}_{n\in\omega}$ and each non-zero element of $Y$ belongs to $\mathrm{L} (\omega)$. Indeed, if $x\coloneqq \sum_{n=0}^\infty \alpha_n e_n \in Y$, then $\mathrm{L}_x=\{\alpha_n \cdot a_k\}_{n,k \in\omega} \cup \{0\}$ (since both $\alpha_n$ and $a_k$ tend to $0$). If additionally $x\in Y \setminus\{0\}$, then some $\alpha_n$ is non-zero, whence $|\mathrm{L}_x|=\omega$, as desired.

For the case of $\mathrm{L}(\mathfrak{c})$, we replace the sequence $a_k=2^{-k}$ with an enumeration $(a_k)_{k\in\omega}$ of the rationals in $(0,1)$. The definition of the vectors $e_n$ is the same with the unique difference that the series defining $e_n$ only converges in the pointwise topology. Now the subspace $Y$ is defined as
\begin{equation*}
    Y\coloneqq \left\{ \sum_{n=1}^\infty \alpha_n e_n\colon (\alpha_n)_{n\in\omega} \in \ell_\infty \right\}
\end{equation*}
(where, as before, the series converges pointwise). Since the vectors $e_n$ are disjointly supported unit vectors, the map $(\alpha_n) _{n\in\omega} \mapsto \sum_{n=1}^\infty \alpha_n e_n$ defines an isometry of $\ell_\infty$ onto $Y$. Finally, as before, we see that $Y\setminus \{0\}\subseteq \mathrm{L}(\mathfrak{c})$, since here $\mathrm{L}_{e_n}=[0,1]$.

\end{proof}

\section{Finer lineability results}\label{sec: finer}
In this section we delve deeper into lineability results for the set $\bigcup_{n \in A}\mathrm{L}(n)$, where $A$ is a (finite) subset of $\omega$ such that $\min A\geq 2$. In the first result we prove Proposition \ref{prop: comb lin many values}, whose validity was claimed during the proof of Theorem \ref{thm: interval not lin}. Next, we give some results that show how more complicated the situation is when $A$ is not an interval. In particular, there are infinite sets $A$ such that $\bigcup_{n \in A}\mathrm{L}(n)$ is not $2$-lineable (Corollary \ref{cor:infinitesetnot2lineable}) and, on the other hand, there are sets $A$ that do not contain non-trivial intervals and such that $\bigcup_{n \in A}\mathrm{L}(n)$ is $\mathfrak{c}$-lineable (Theorem \ref{th: odd c lineable}).

\begin{prop}\label{prop: comb lin many values} Let $n,d\in\omega$ with $n\geq 1$. Then there are vectors $\{v_1,\dots,v_{n+d}\}\in \R^{d+1}$ such that, for all non-zero $\alpha\coloneqq (\alpha_0,\dots,\alpha_d)\in \R^{d+1}$, the set $\{\alpha\cdot v_k\}_{k=1}^{n+d}$ has cardinality at least $n$.
\end{prop}
We recall that we indicate by $\alpha\cdot v$ the inner product of the vectors $\alpha, v\in \R^{d+1}$. Note that, by Lemma \ref{lem: decrease lims}, there is a non-zero $\alpha\coloneqq (\alpha_0,\dots,\alpha_d)\in \R^{d+1}$ such that $\{\alpha\cdot v_k\}_{k=1}^{d+1}$ is a singleton. Hence, for such $\alpha$ the set $\{\alpha\cdot v_k\} _{k=1}^{n+d}$ has cardinality at most $n$, so the above result is sharp.

\begin{proof} The result is trivial for $n=1$, thus we assume that $n\geq 2$.  We begin by introducing a piece of notation. Assume that $\mathcal{V}=\{v_1,\dots,v_k\}$ (where $k\geq 1$) are vectors in $\R^{d+1}$ and $\{\mathcal{V}_1, \dots, \mathcal{V}_{n-1} \}$ is a partition of $\mathcal{V}$ in exactly $n-1$, possibly empty, sets. For every $j\in\{1, \dots, n-1\}$ such that $\mathcal{V}_j$ is non-empty we define a vector $w_j\in \mathcal{V}_j$ to be $w_j\coloneqq v_i$, where $i$ is the least index with $v_i\in \mathcal{V}_j$. Roughly speaking, $w_j$ is the `first' vector in $\mathcal{V}_j$. Moreover, we define sets
$$\mathcal{W}_j\coloneqq \big\{ v-w_j\colon v\in \mathcal{V}_j\setminus\{w_j\} \big\}, \,\,\text{ when }\,\, \mathcal{V}_j\neq \emptyset$$
and $\mathcal{W}_j=\emptyset$ otherwise. Finally, we say that the set $\mathcal{V}=\{v_1,\dots,v_k\}$ has the \emph{many increments property} (MIP, for short) if for every partition $\mathscr{V}=\{\mathcal{V}_1, \dots, \mathcal{V}_{n-1}\}$ of $\mathcal{V}$:
\begin{itemize}
    \item[(MIP1)] the sets $\{\mathcal{W}_j\}_{j=1}^{n-1}$ are pairwise disjoint, and
    \item[(MIP2)] setting $\mathcal{W}\coloneqq\bigcup_{j=1}^{n-1}\mathcal{W}_j$, $\Span\mathcal{W}$ has dimension at least $\min\{|\mathcal{W}|,d+1$\}.
\end{itemize}

\begin{claim}\label{claim: MIP} There exists a family $\mathcal{V}=\{v_1,\dots, v_{n+d}\}\subseteq \R^{d+1}$ consisting of mutually distinct vectors and having property (MIP).
\end{claim}

Assuming the claim for now, let us show that a family as in the claim also verifies the conclusion of the proposition. In fact, given such a $\mathcal{V}$, for any partition $\mathscr{V}=\{\mathcal{V}_1, \dots, \mathcal{V}_{n-1}\}$ of $\mathcal{V}$, the above set $\mathcal{W}$ satisfies $\Span(\mathcal{W})=\R^{d+1}$. Indeed, by (MIP1)
\begin{equation*}
    \left|\mathcal{W}\right| = \sum_{j=1}^{n-1} \left|\mathcal{W}_j\right| \geq \sum_{j=1}^{n-1} (|\mathcal{V}_j|-1)
    = |\mathcal{V}|-(n-1) = d+1.
\end{equation*}
So, $\Span(\mathcal{W})$ has dimension $d+1$ by (MIP2). Suppose now that $\alpha\coloneqq (\alpha_0,\dots, \alpha_d)\in \R^{d+1}$ is such that $\{\alpha\cdot v_k\}_{k=1}^{n+d}$ has cardinality at most $n-1$. Then there is a partition $\{\mathcal{V}_1, \dots, \mathcal{V}_{n-1}\}$ of $\mathcal{V}$ such that $\{\alpha\cdot v\colon v\in \mathcal{V}_j\}$ is at most a singleton for every $j\in\{1,\dots,n-1\}$ (in order to have exactly $n-1$ elements in the partition, some $\mathcal{V}_j$ might be empty). But this means that
\begin{equation*}
    \alpha\cdot (v-w_j)=0 \text{ for all } j \text{ such that } \mathcal{V}_j\neq \emptyset \,\,\,\text{and all } v\in \mathcal{V}_j\setminus\{w_j\}.
\end{equation*}
In other words, $\alpha$ is orthogonal to all the vectors in $\mathcal{W}$. Therefore, $\alpha$ is orthogonal to $\Span(\mathcal{W})$, which by our construction is equal to $\R^{d+1}$; thus $\alpha=0$, as desired. \smallskip

Therefore, we only need to prove Claim \ref{claim: MIP} and we build the vectors $\{v_1,\dots, v_{n+d}\}$ recursively (recall that $n$ and $d$ are fixed). Set $v_1\coloneqq 0$ and note that, up to relabelling, the unique partition $\{\mathcal{V}_1, \dots, \mathcal{V}_{n-1}\}$ of $\{v_1\}$ is given by $\mathcal{V}_1 =\{v_1\}$ and $\mathcal{V}_2 =\dots= \mathcal{V}_{n-1}= \emptyset$. Hence, $\mathcal{W}_j= \emptyset$ for every $j$, so the singleton $\{v_1\}$ satisfies (MIP). Suppose now that, for some $k\leq n+d-1$, we have already found vectors $\{v_1, \dots, v_k\}$ satisfying property (MIP). We now look for conditions on $v_{k+1}$ so that the property (MIP) holds also for $\{v_1,\dots,v_{k+1}\}$. First of all, we need $v_{k+1}\notin \{v_1,\dots,v_k\}$. Next, assume that $\mathscr{V}=\{\mathcal{V}_1, \dots, \mathcal{V}_{n-1}\}$ is a partition of $\{v_1,\dots,v_{k+1}\}$ and, up to relabelling the indices of the partition, that $v_{k+1}\in \mathcal{V}_{n-1}$. If $\mathcal{V}_{n-1}= \{v_{k+1}\}$, then $\mathcal{W}_{n-1}= \emptyset$, so (MIP1) and (MIP2) are satisfied because of the inductive assumption applied to the partition $\{\mathcal{V}_1, \dots, \mathcal{V}_{n-2}, \emptyset\}$ of $\{v_1, \dots, v_k\}$. 

Therefore, we assume that $\mathcal{V}_{n-1}$ is not a singleton, whence $w_{n-1}\neq v_{k+1}$, by definition of $w_{n-1}$. In order to satisfy condition (MIP1) for the partition $\mathscr{V}$, the vector $v_{k+1} - w_{n-1}$ should not belong to $\mathcal{W}_j$ for every $j\in \{1,\dots,n-2\}$. Since there are only finitely many partitions $\{\mathcal{V}_1, \dots, \mathcal{V}_{n-1}\}$ of $\{v_1,\dots,v_{k+1}\}$, we conclude that $v_{k+1}$ must be chosen outside a finite subset of $\R^{d+1}$. In order to verify (MIP2), we distinguish two cases. If the vectors in
\begin{equation}\label{eq: d or d+1}
\mathcal{W}_*\coloneqq \bigcup_{j= 1}^{n-2}\mathcal{W}_j \cup \big\{ v-w_{n-1} \colon v\in \mathcal{V}_{n-1} \setminus\{w_{n-1},v_{k+1}\} \big\}
\end{equation}
are at least $d+1$ in number, then their linear span has dimension at least $d+1$, by the (MIP2) property of $\{v_1, \dots, v_k\}$. \emph{A fortiori}, $\Span\mathcal{W}$ has dimension at least $d+1$, so no condition is imposed on $v_{k+1}$.

Otherwise, suppose that $\mathcal{W}_*$ has cardinality at most $d$. Therefore, the linear span of $\mathcal{W}_*$ is a proper subspace $H$ of $\R^{d+1}$, of dimension exactly $\left|\mathcal{W}_*\right|$ by (MIP2). Moreover, $\mathcal{W}= \mathcal{W}_*\cup \{v_{k+1} -w_{n-1}\}$. Hence, the vectors $\{v_1,\dots, v_{k+1}\}$ satisfy (MIP2) if and only if $v_{k+1}-w_{n-1}$ is linearly independent from $H$. In other words, if and only if $v_{k+1}$ does not belong to the proper affine subspace $w_{n-1}+H$. Consequently, since there are only finitely many partitions $\{\mathcal{V}_1,\ldots,\mathcal{V}_{n-1}\}$ of $\{v_1,\ldots,v_{k+1}\}$, then the vector $v_{k+1}$ must be chosen outside finitely many proper affine subspaces of $\R^{d+1}$. This yields that it is possible to select $v_{k+1}\notin \{v_1,\dots,v_k\}$ such that $\{v_1,\ldots,v_{k+1}\}$ satisfies property (MIP) and concludes the proof.
\end{proof}

For the second part of the section, it will be convenient to introduce the following notation. For each non-empty set $A\subseteq \omega$ with $\min A\geq 2$, define
\begin{displaymath}
\ell(A)\coloneqq\sup\left\{m \in \omega\colon \bigcup_{n \in A}\mathrm{L}(n) \text{ is }m\text{-lineable}\right\}.
\end{displaymath}

Note that Theorem \ref{thm: interval not lin} can be equivalently rewritten as $\ell(A)=|A|$ whenever $A\subseteq\omega$ is a finite non-empty interval with $\min A\geq 2$.
The same theorem also implies \begin{equation}\label{eq: l(A) larger than interval}
    \ell(A)\geq \sup\{|I|\colon I\subseteq A \text{ is an interval}\}
\end{equation}
whenever $A\subseteq \omega$ is a non-empty set  with $\min A\geq 2$. We will see in the forthcoming results that, when $A$ is not an interval, the inequality \eqref{eq: l(A) larger than interval} can be very far from being sharp. Indeed, we will see in Theorem \ref{th: odd c lineable} that there exist sets $A$ with $\ell(A)=\infty$ and which contain no non-trivial intervals. Before this, we prove the existence of infinite sets $A$ such that $\ell(A)=1$ (see Corollary \ref{cor:infinitesetnot2lineable}).

\begin{lem}\label{lem:improvedloweraccpoints} Fix vectors $x \in \mathrm{L}(n)$ and $y \in \mathrm{L}(k)$, for some $n,k \in \omega$, with representations
\begin{equation*}
    x \sim_{c_0} \xi_1 \bm{1}_{S_1}+\dots+\xi_n \bm{1}_{S_n} \quad \text{ and }\quad y \sim_{c_0} \eta_1 \bm{1}_{T_1}+\dots+\eta_k \bm{1}_{T_k},
\end{equation*}
respectively, as in Lemma \ref{lem: mod c0}. Define 
\begin{equation*}
    \mathcal{E}\coloneqq \left\{(i,j) \in \{1,\ldots,n\} \times \{1,\ldots,k\}\colon S_i\cap T_j \text{ is infinite}\right\}
\end{equation*}
and suppose that the points in 
\begin{equation*}
    \mathcal{P}\coloneqq \left\{(\xi_i,\eta_j)\colon (i,j) \in \mathcal{E}\right\}   
\end{equation*}
are not collinear. Then there is $z\in \Span\{x,y\}$ such that
\begin{equation*}
    \left\lfloor \frac{|\mathcal{E}|+1}{2}\right\rfloor \leq |\mathrm{L}_z| \leq |\mathcal{E}|-1.
\end{equation*}
\end{lem}

\begin{proof} We begin with the following combinatorial observation. Let $\mathcal{P}$ be a set of $m$ non-collinear points in the plane. Then there exists a line $\ell$, determined by at least two points in $\mathcal{P}$, such that if $\mathscr{L}$ is a set of parallel lines to $\ell$ and $\mathcal{P}\subseteq \mathscr{L}$, then $|\mathscr{L}|\geq\left\lfloor \frac{m+1}{2}\right\rfloor$. Indeed, according to \cite{MR2399380}, there are a line $\ell$, determined by at least two points of $\mathcal{P}$, and $\left\lfloor \frac{m-1}{2} \right\rfloor$ points in $\mathcal{P}$ whose distances from $\ell$ are positive and mutually distinct (see the first sentence in \cite[\S\ 2]{MR2399380}). Such points necessarily belong to mutually distinct lines from $\mathscr{L} \setminus \{\ell\}$, so $|\mathscr{L}| \geq\left\lfloor \frac{m-1}{2}\right\rfloor+1= \left\lfloor \frac{m+1}{2} \right\rfloor$.

Now let $\mathcal{P}\coloneqq \left\{(\xi_i,\eta_j)\colon (i,j) \in \mathcal{E}\right\}$ and note that $|\mathcal{P}|= |\mathcal{E}|$, since the $\xi_i$'s and the $\eta_j$'s are mutually distinct. Let $\ell$ be a line as in the observation above; then there are scalars $\alpha,\beta,\gamma\in\R$ such that $\ell =\{(\xi,\eta)\in\R^2\colon \alpha\xi + \beta\eta = \gamma\}$. 
Therefore 
\begin{equation*}
    \left\lfloor \frac{|\mathcal{E}|+1}{2}\right\rfloor \leq \big|\left\{\alpha\xi_i + \beta\eta_j\colon (i,j) \in \mathcal{E}\right\}\big| \leq |\mathcal{E}|-1,
\end{equation*}
the right-hand side inequality being true because two distinct points of $\mathcal{P}$ belong to $\ell$. The conclusion follows observing that $\mathrm{L}_{\alpha x+\beta y}=\left\{\alpha\xi_i + \beta\eta_j\colon (i,j) \in \mathcal{E}\right\}$.
\end{proof}

\begin{prop}\label{prop:nottoolargegaps}
Fix a non-empty finite set $A\subseteq \omega$ with $\min A\geq 2$ and fix $k \in \omega$ such that $k> 2\max A$. 
Then $\ell(A\cup \{k\})=\ell(A)$. 
\end{prop}

\begin{proof} Assume, towards a contradiction, that $\ell(A)< \ell(A\cup \{k\})$ and take a subspace $V$ of $\bigcup_{i \in A\cup \{k\}}\mathrm{L}(i)\cup \{0\}$ of dimension $\ell(A)+1$. By definition of $\ell(A)$, there exists a vector $y\in V \cap\mathrm{L}(k)$. Moreover, since $\ell(A)+1\geq 2$ and $\mathrm{L}(k)$ is not $2$-lineable by Theorem \ref{thm: interval not lin}, we can take a vector $x\in V$ such that $|\mathrm{L}_x|\in A$. Hence, letting $M\coloneqq \max A$, we have $|\mathrm{L}_x|\leq M$. Thanks to Lemma \ref{lem: mod c0}, we have the representations
\begin{equation*}
    x\sim_{c_0} \xi_1 \bm{1}_{S_1} + \dots+ \xi_n \bm{1}_{S_n} \quad \text{ and }\quad y\sim_{c_0} \eta_1 \bm{1}_{T_1} + \dots + \eta_{k} \bm{1}_{T_{k}}
\end{equation*}
(here $n\coloneqq |\mathrm{L}_x|\leq M$). 

Consider the sets $\mathcal{E}$ and $\mathcal{P}$ corresponding to $x$ and $y$ as in the statement of Lemma \ref{lem:improvedloweraccpoints}. If all points in $\mathcal{P}$ belong to the same line $\{(\xi,\eta)\in \R^2\colon \alpha\xi + \beta\eta=\gamma\}$, then the sequence $\alpha x + \beta y\in V$ would be convergent to $\gamma$, a contradiction. Hence the points of $\mathcal{P}$ are not collinear. Moreover, by our assumption on $V$, every linear combination of $x$ and $y$ has at most $k$ accumulation points. Thus, by Lemma \ref{lem: supporti inscatolati}, there exists a partition $\{I_1,\dots,I_n\}$ of $\{1,\dots, k\}$ such that $S_j=^*\bigcup_{i \in I_j}T_i$ for $j \in \{1,\dots,n\}$. This assures us that $|\mathcal{E}|=k$. Therefore, we can apply Lemma \ref{lem:improvedloweraccpoints} and we obtain the existence of a vector $z\in\Span\{x,y\}$ such that 
\begin{equation*}
    \left\lfloor \frac{k+1}{2}\right\rfloor \leq |\mathrm{L}_z| \leq k-1.
\end{equation*}
However, the assumption $k> 2\max A$ implies $ \left\lfloor \frac{k+1}{2}\right\rfloor> \max A$, so $\mathrm{L}_z\notin A\cup\{k\}$, a contradiction with the fact that $\mathrm{span}\{x,y\} \subseteq V \subseteq \bigcup_{i \in A\cup \{k\}}\mathrm{L}(i)\cup \{0\}$.
\end{proof}

\begin{rmk} The above proof shows that, if $A$ and $k$ are as in the statement of Proposition \ref{prop:nottoolargegaps}, then every vector space contained in $\bigcup_{n \in A\cup \{k\}} \mathrm{L}(n)\cup \{0\}$ and of dimension at least $2$ does not intersect $\mathrm{L}(k)$. This is not true anymore if $A$ and $k$ don't satisfy the condition of the proposition, as the following example shows.
\end{rmk}

\begin{example}\label{rmk:strangeoctagon}
For each integer $n\geq 2$, set $A_n\coloneqq \{n,n+1,2n\}$ and take vectors $x,y \in \mathrm{L}(2n)$ with
\begin{equation*}
    x \sim_{c_0} \xi_1 \bm{1}_{S_1} +\dots+\xi_{2n} \bm{1}_{S_{2n}} \quad \text{ and }\quad y \sim_{c_0} \eta_1 \bm{1}_{S_1}+\dots +\eta_{2n} \bm{1}_{S_{2n}},
\end{equation*}
where $\{S_1,\dots, S_{2n}\}$ is a partition of $\omega$ into infinite sets. Further, the two families of distinct scalars $\{\xi_1,\dots, \xi_{2n}\}$ and $\{\eta_1,\dots, \eta_{2n}\}$ are chosen so that, if $P_j \coloneqq(\xi_j,\eta_j)$, then $\mathcal{P}=\{P_1, \dots, P_{2n}\}$ are the vertices of a regular polygon with $2n$ edges labelled in the clockwise order. Then $V\coloneqq \Span\{x,y\}$ is a $2$-dimensional vector space such that $V\subseteq \bigcup_{k \in A_n}\mathrm{L}(k)\cup \{0\}$ and $V\cap \mathrm{L}(k) \neq \emptyset$ for each $k \in A_n$. Indeed, if $\mathscr{L}$ is a set of parallel lines such that $\mathcal{P}\subseteq \mathscr{L}$ and every line in $\mathscr{L}$ contains a point of $\mathcal{P}$, then there are three cases. If every line in $\mathscr{L}$ only contains one point of $\mathcal{P}$, then $|\mathscr{L}|=2n$; if one line in $\mathscr{L}$ contains $P_1$ and $P_2$, then $|\mathscr{L}|=n$; finally, if one line in $\mathscr{L}$ contains $P_1$ and $P_3$, then $|\mathscr{L}|=n+1$. We omit the elementary geometric considerations required to prove that there only are these three cases (and we advise the reader to draw a picture).
\end{example}

We now give two examples of consequences of the above result. The first one implies in particular that, if $A$ is not an interval, Corollary \ref{cor: lims in A not lin} might not be sharp:
\begin{cor}\label{cor:LnLnplusdnot2lineable}
Let $n,k \in \omega$ be such that $n\geq 2$ and $k>2n$. Then $\mathrm{L}(n)\cup \mathrm{L}(k)$ is not $2$-lineable.
\end{cor}
\begin{proof} We have $\ell(\{n\})=1$ by Theorem \ref{thm: interval not lin}. Since $k>2n$, we conclude by Proposition \ref{prop:nottoolargegaps} that $\ell(\{n,k\})=\ell(\{n\})=1$.
\end{proof}

By iteration of the above argument, we readily obtain the following result. It implies in particular that $\bigcup_{2\leq n<\omega} \mathrm{L}(n!)$ and $\bigcup_{1\leq n<\omega} \mathrm{L}(3^n)$ are not $2$-lineable.

\begin{cor}\label{cor:infinitesetnot2lineable}
Let $(a_n)_{n \in \omega}$ be an increasing sequence in $\omega$ such that $a_0\geq 2$ and $a_{n+1}>2a_{n}$ for all $n\in \omega$. Then $\bigcup_{n \in \omega}\mathrm{L}(a_n)$ is not $2$-lineable.
\end{cor}
\begin{proof} Applying inductively Proposition \ref{prop:nottoolargegaps} to $\{a_0,\dots, a_n\}$ and $a_{n+1}$ we obtain that $\bigcup_{k\leq N} \mathrm{L}(a_k)$ is not $2$-lineable for every $N\in\omega$. If there exists a $2$-dimensional vector space $V\subseteq \bigcup_{k \in \omega}\mathrm{L}(a_k)\cup\{0\}$, then, by Lemma \ref{lem:boundonnumberaccpoints}, $V\subseteq \bigcup_{k\leq N}\mathrm{L}(a_k)\cup\{0\}$ for some $N$, a contradiction.
\end{proof}

These type of results and \eqref{eq: l(A) larger than interval} might lead one to conjecture that $\ell(A)$ could be large only if $A$ contains large intervals. The last result of the section gives a strong negative answer to this conjecture. In particular, it follows that the inequality \eqref{eq: l(A) larger than interval} is not sharp, even if $A$ is finite.

\begin{thm}\label{th: odd c lineable} $\bigcup_{1\leq n<\omega}\mathrm{L}(2n+1)$ is $\mathfrak{c}$-lineable.
\end{thm}

\begin{proof} Let $\mathscr{A}\coloneqq \{A_{\gamma}^{e} \colon e\in\{-1,0,1\}, \gamma\in\mathfrak{c}\} \subseteq \mathcal{P} (\omega)$ be such that $\{A_{\gamma}^{-1}, A_\gamma^{0}, A_\gamma^{1}\}$ is a partition of $\omega$ for each $\gamma\in\mathfrak{c}$ and
\begin{equation*}
    A_{\gamma_1}^{e_1} \cap \dots \cap A_{\gamma_k}^{e_k} \text{ is an infinite set}
\end{equation*}
for all $k\geq 1$, all distinct $\gamma_1,\dots,\gamma_k\in \mathfrak{c}$, and all $e=(e_j)_{j=1}^{k} \in \{-1,0,1\}^k$.
Let us observe that the existence of such a family $\mathscr{A}$ follows similarly as the existence of independent families (see, \emph{e.g.}, \cite[Example 2, p. 10]{MR3698923}). Indeed, the set $\Q[x]$ of polynomials with rational coefficients is countable, hence we can construct $\mathscr{A}$ as a subset of $\mathcal{P}(\Q[x])$. Therefore, it is sufficient to define, for each $\gamma \in \R$, $A_{\gamma}^{-1} \coloneqq \{p \in P \colon p(\gamma)\leq -1\}$, $A_{\gamma}^{0}\coloneqq \{p \in P\colon |p(\gamma)|<1\}$, and $A_{\gamma}^{1}\coloneqq \{p \in P\colon p(\gamma)\geq 1\}$.

At this point, for each $\gamma\in \mathfrak{c}$, define the vector 
\begin{equation}\label{eq: odddefinition}
  x_\gamma\coloneqq \bm{1}_{A_{\gamma}^{1}}-  \bm{1}_{A_{\gamma}^{-1}}
\end{equation}
and set $V\coloneqq \Span\{x_\gamma\colon \gamma\in \mathfrak{c}\}$. We claim that each non-zero $z \in V$ is a non-convergent sequence with an odd number of accumulation points. To this aim, suppose that $z=\sum_{j=1}^k \alpha_j x_{\gamma_j}$ for some non-zero $\alpha_1,\dots,\alpha_k \in \mathbb{R}$ and some distinct $\gamma_1,\dots,\gamma_k\in \mathfrak{c}$.

Note that
\begin{equation*}
    \mathrm{L}_z =\left\{\sum_{j=1}^k \alpha_j e_j\colon e=(e_j)_{j=1}^k \in \{-1,0,1\}^k \right\}.
\end{equation*}
Since $-e\in \{-1,0,1\}^k$ whenever $e\in \{-1,0,1\}^k$, we obtain that $\mathrm{L}_z= -\mathrm{L}_z$. Moreover, setting $e=(e_1,0, \dots, 0)$ with $e_1 \in \{-1,0,1\}$, we get that $\alpha_1\cdot \{-1,0,1\} \subseteq \mathrm{L}_z$. Hence $|\mathrm{L}_z|\geq 3$ and $0 \in \mathrm{L}_z$. Combining this with $\mathrm{L}_z= -\mathrm{L}_z$, the result follows.
\end{proof}

\begin{rmk}\label{rmk:limitones} The same argument as above, paired with Lemma \ref{lem:boundonnumberaccpoints}, proves that $\Span\{x_0,x_1\}$ is a $2$-dimensional vector space contained in $\bigcup_{n \in A}\mathrm{L}(n) \cup \{0\}$, where $A\coloneqq \{3,5,7,9\}$ and the vectors $x_i$ are defined as in \eqref{eq: odddefinition}. Therefore $\ell(\{3,5,7,9\}) \geq 2$.
\end{rmk}

\section{Final remarks}\label{sec: problems}
In this last section we collect some observations concerning possible improvements of the results presented in our paper. Let us start with one comment concerning maximal lineability. A subset $M$ of a vector space $X$ is \emph{maximal lineable} if $M\cup\{0\}$ contains a linear subspace $V$ such that $\dim(V)=\dim(X)$. Clearly, every dense subspace $V$ of $\ell_\infty$ satisfies $\dim(V)=\mathfrak{c}$, merely because $\dim(\ell_\infty) =\dens(\ell_\infty) =\mathfrak{c}$. Consequently, all our results concerning dense lineability in $\ell_\infty$ automatically are `maximal dense lineability' results (note that the situation is different if the Banach space $X$ is separable, since maximal lineability would require finding a subspace of dimension continuum, while a dense subspace might have countable dimension). 

\subsection{Ideal convergence}
Next, we discuss extensions of our results to the setting of ideal convergence. Recall that an ideal $\mathcal{I}$ on $\omega$ is a proper subfamily of $\mathcal{P}(\omega)$ which is closed under subsets and finite unions and that contains all singletons of $\omega$. We denote by $\mathrm{Fin}$ the ideal of finite sets; hence $\mathrm{Fin}\subseteq \mathcal{I}$ for every ideal $\mathcal{I}$. For each sequence $x \in \ell_\infty$, let $\Gamma_x(\mathcal{I})$ be the set of its $\mathcal{I}$-cluster points, \emph{i.e.}, the set of all $\eta \in \mathbb{R}$ such that $\{n \in \omega\colon |x(n)-\eta|<\e\}\notin \mathcal{I}$ for all $\e>0$. It is easy to see each $\Gamma_x(\mathcal{I})$ is non-empty, closed, and contained in $\mathrm{L}_x=\Gamma_x(\mathrm{Fin})$. 
Given a cardinal $\kappa$, define the set
\begin{equation*}
    \Gamma(\mathcal{I}, \kappa)\coloneqq \left\{x \in \ell_\infty\colon |\Gamma_x(\mathcal{I})|= \kappa\right\}.
\end{equation*}
Hence $\Gamma(\mathrm{Fin},\kappa)=\mathrm{L}(\kappa)$ for all $\kappa$ and $\Gamma(\mathcal{I},\kappa)= \emptyset$ for uncountable $\kappa<\mathfrak{c}$. 

Let $\mathcal{I}$ be an ideal on $\omega$ such that there are disjoint subsets $(B_j)_{j\in\omega}$ of $\omega$ with $B_j\notin \mathcal{I}$ for every $j\in\omega$. We stress here that this condition is satisfied by a large class of ideals. Besides the case $\mathcal{I}= \mathrm{Fin}$, it holds for all meagre ideals, as it readily follows from a classical characterisation of meagre filters due to Talagrand \cite[Th\'eor\`eme 21]{Talagrand}, see also \cite[Theorem 4.1.2]{MR1350295}. Moreover, this condition is satisfied by all ideals which do not contain any isomorphic copy of a maximal ideal. Then, minimal variations in the proofs of Theorem \ref{th: L omega dense lin}, Remark \ref{rmk: Lc dense lin}, and Theorem \ref{thm:Lomegaspaceable} give that both $\Gamma(\mathcal{I},\omega)$ and $\Gamma (\mathcal{I}, \mathfrak{c})$ are densely lineable in $\ell_\infty$ and spaceable. We chose to state our main results only for $\mathcal{I}=\mathrm{Fin}$ for the sake of clarity of the exposition, but we now quickly discuss how to prove the more general case.

The spaceability results are obtained from Theorem \ref{thm:Lomegaspaceable} by using a partition $(A_{j,k})_{j,k\in\omega}$ such that $A_{j,k}\notin \mathcal{I}$. The dense lineability of $\Gamma(\mathcal{I},\omega)$ in $\ell_\infty$ follows from the very same argument as in Theorem \ref{th: L omega dense lin}, using again a partition $(B_j)_{j\in\omega}$ such that $B_j\notin\mathcal{I}$. Note that $X\coloneqq \bigcup_{\kappa\leq \omega} \Gamma(\mathcal{I},\kappa)$ and $Y\coloneqq \bigcup_{\kappa< \omega} \Gamma (\mathcal{I},\kappa)$ are still vector spaces, due to the standard fact that 
\begin{equation}\label{eq: I cluset of sum}
    \Gamma_{x+y}(\mathcal{I})\subseteq \Gamma_x(\mathcal{I})+\Gamma_y(\mathcal{I})
\end{equation}
(which readily follows, \emph{e.g.}, from \cite[Chapter I, \S\ 7, no.\ 3, Proposition 8]{Bourbaki}). For the dense lineability of $\Gamma(\mathcal{I},\mathfrak{c})$ in $\ell_\infty$, we need sequences $r_j\colon \omega \to(0,1)$ with $\suppt(r_j)=B_j$ and $\Gamma_{r_j} (\mathcal{I})=[0,1]$. For this, take disjoint sets $(A_{j,k})_{j,k\in\omega}$ with $A_{j,k}\notin \mathcal{I}$ and define $B_j\coloneqq \cup_{k\in\omega} A_{j,k}$. Then, let $(q_k)_{k\in\omega}$ be an enumeration of $\Q\cap(0,1)$ and define $r_j$ to be equal to $q_k$ on $A_k$ ($k\in\omega$) and $0$ elsewhere. The same argument as in Remark \ref{rmk: Lc dense lin}, using again \eqref{eq: I cluset of sum}, gives the result.

Finally, we can also modify the proof of Theorem \ref{thm:finitepart} to prove that $\bigcup_{2\leq n< \omega} \Gamma (\mathcal{I},n)$ is densely lineable in $\ell_\infty$, for the same class of ideals. Indeed, let $(B_j)_j\in\omega$ be a partition of $\omega$ as above and let $\mathscr{I}$ be an independent family on $\omega$ of cardinality $\mathfrak{c}$. For $A\in\mathscr{I}$ define $B_A\coloneqq \bigcup_{j\in A} B_j$ and let $\mathscr{J}\coloneqq\{B_A\colon A\in\mathscr{I}\}$. Then $V\coloneqq \Span\{\bm{1}_B\colon B\in \mathscr{J}\}$ is a vector space of dimension $\mathfrak{c}$, every vector in $V$ has finitely many $\mathcal{I}$-cluster points, and $V\cap \Gamma(\mathcal{I},1)=\{0\}$. To prove the last assertion, let $D_0,\dots, D_N\in\mathscr{J}$ and non-zero scalars $d_0,\dots,d_N\in\R$. By definition, the sets
\begin{equation*}
    D_0 \setminus \left(D_1 \cup\dots\cup D_N\right) \quad \text{and} \quad \omega \setminus \left(D_0 \cup\dots\cup D_N\right)
\end{equation*}
do not belong to $\mathcal{I}$. Hence $\sum_{j=0}^N d_j \bm{1}_{D_j}$ attains the values $d_0$ and $0$ on sets that do not belong to $\mathcal{I}$, thus it is not $\mathcal{I}$-convergent.

\subsection{\texorpdfstring{$\R^\omega$}{Rω} and pointwise convergence}
Although all the paper remained in the realm of Banach spaces, we only used little Banach space theoretic structure of $\ell_\infty$. Therefore, it is natural to ask whether similar results can be true if we replace $\ell_\infty$ with the larger space $\R^\omega$ of all scalar sequences. For a sequence $(x(n))_{n\in\omega} \in\R^\omega$, the set $\mathrm{L}_x$ of accumulation points of $x$ is now defined as a subset of $\R\cup \{\pm\infty\}$ (if a subsequence of $(x(n))_{n\in\omega}$ diverges to $\pm\infty$, $\pm\infty$ is considered to be an accumulation point of the sequence). The definition of $\mathrm{L}(\kappa)$ is also modified accordingly; for example, every sequence $(x(n))_{n\in\omega}$ that diverges to $\infty$ belongs to $\mathrm{L}(1)$.

$\R^\omega$ is a separable, completely metrisable topological vector space when endowed with the pointwise topology. (Throughout all the subsection, we always endow $\R^\omega$ with the pointwise topology.) In addition, $c_{00}$ (and, hence, $\ell_\infty$) is dense in $\R^\omega$. Therefore, our results immediately imply that $\mathrm{L}(\mathfrak{c})$, $\mathrm{L}(\omega)$, and $\bigcup_{2 \leq n<\omega}\mathrm{L}(n)$ are densely lineable in $\R^\omega$. Note, on the other hand that the results for $\ell_\infty$ are stronger, since the norm topology is substantially finer than the pointwise one; in particular, there is no obvious way to recover the results for $\ell_\infty$ from the corresponding one for $\R^\omega$.

As regards spaceability, the same argument as in Theorem \ref{thm:Lomegaspaceable} shows that $\mathrm{L} (\mathfrak{c})$ is spaceable in $\R^\omega$. Indeed, if $Y\coloneqq\closedSpan\{e_n\} _{n\in\omega}$, where $e_n$ is as in \eqref{eq: en for spaceable} and the closure is in the pointwise topology, then
\begin{equation*}
    Y=\left\{\sum_{n=0}^\infty \sum_{k=0}^\infty \alpha_n a_k\cdot \bm{1}_{A_{n,k}}\colon (\alpha_n)_{n\in \omega} \in\R^\omega \right\}
\end{equation*}
(here, the above series converge in the pointwise topology). Hence, if $x\in Y\setminus\{0\}$, write $x\coloneqq \sum_{n=0}^\infty \sum_{k=0}^\infty \alpha_n a_k\cdot \bm{1}_{A_{n,k}}$ and take $n\in\omega$ with $\alpha_n\neq 0$; thus $\alpha_n\cdot [0,1]\subseteq \mathrm{L}_x$.

On the other hand, the above argument does not extend to prove that $\mathrm{L}(\omega)$ is spaceable in $\mathbb{R}^\omega$ (because the sequence $(\alpha_n)_{n\in \omega} \in \R^\omega$ can create uncountably many accumulation points). Interestingly, it turns out that, differently from Theorem \ref{thm:Lomegaspaceable}, $\mathrm{L}(\omega)$ is not spaceable in $\mathbb{R}^\omega$. In the next theorem we actually prove a more general result.

\begin{thm} For every closed infinite-dimensional subspace $Y$ of $\R^\omega$ there is $x\in Y$ such that $\mathrm{L}_x= \R\cup\{\pm \infty\}$. In particular, $\bigcup_{\kappa\leq \omega} \mathrm{L}(\kappa)$ is not spaceable in $\R^\omega$.
\end{thm}
Let us remark that, aside implying the non spaceability of $\mathrm{L}(\omega)$, the result implies that also $\bigcup_{2 \leq n<\omega}\mathrm{L}(n)$ is not spaceable in $\R^\omega$.

\begin{proof} Notice that $\{x\in Y\colon n\leq \min (\suppt(x))\}$ has finite codimension in $Y$ for every $n\in\omega$. Hence, by the fact that $Y$ is infinite-dimensional, we can find a sequence $(x_n)_{n\in\omega}$ of non-zero vectors of $Y$ such that the sequence $s_n\coloneqq \min (\suppt(x_n))$ is strictly increasing. Then, let $(q_n)_{n\in\omega}$ be an enumeration of $\Q$. Take $(\alpha(n))_{n\in\omega} \in\R^\omega$ that solves the following system of equations:
\begin{equation*}
    \sum_{j=0}^k \alpha(j)x_j(s_k)=q_k \qquad (k\in\omega).
\end{equation*}
Such a system can indeed be solved recursively, using the fact that $x_k(s_k)\neq 0$ for every $k\in\omega$. We are now in position to define the vectors
\begin{equation*}
    u_k\coloneqq \sum_{j=0}^k \alpha(j)x_j\in Y \quad \text{ and }\quad u\coloneqq \sum_{j=0}^\infty \alpha(j)x_j.
\end{equation*}
Note that the series defining $u$ converges pointwise, due to the assumption that $(s_n)_{n\in\omega}$ is strictly increasing. By the same reason, we also conclude that $u_k\to u$ pointwise, hence $u\in Y$. However, $u(s_k)=q_k$ for every $k\in\omega$, hence $\mathrm{L}_u= \R\cup\{\pm \infty\}$.
\end{proof}

\subsection{Further research}
In conclusion of our presentation, we shall highlight some directions for possible further research that seem natural in light of the results presented. Concerning spaceability, recall that the two closed subspaces that we constructed in Theorem \ref{thm:Lomegaspaceable}, contained in $\mathrm{L}(\omega) \cup\{0\}$ and $\mathrm{L} (\mathfrak{c}) \cup\{0\}$, are isometric to $c_0$ and $\ell_\infty$ respectively. It would be interesting to know whether it is possible to build a non-separable closed subspace also in the case of $\mathrm{L}(\omega)$.

\begin{probl}\label{prob: non separable} Does $\mathrm{L}(\omega)\cup\{0\}$ contain a closed non-separable subspace? Does it contain an isometric copy of $\ell_\infty$? The same questions could be asked for $\bigcup_{\kappa \leq \omega}\mathrm{L}(\kappa)$.
\end{probl}

Another possible direction of investigation could consist in digging deeper in the linear structure of the sets $\bigcup_{n\in A} \mathrm{L}(n)$, where $A\subseteq \omega$ and $\min A\geq 2$. In Theorem \ref{thm: interval not lin} we gave a complete result in the case when $A$ is an interval of the form $\{n,n+1,\dots, n+d\}$. On the other hand, we saw in Section \ref{sec: finer} that when $A$ is not an interval the situation is less clear. For example, it is quite conceivable that the assumption $k> 2\max A$ in Proposition \ref{prop:nottoolargegaps} could be improved. In the same direction, one might try to characterise those finite sets $A$ for which $\ell(A)=\max\{|I|\colon I\subseteq A \text{ is an interval}\}$.

A slightly different question, that we find particularly interesting, is the following (which ought to be compared with Theorem \ref{th: odd c lineable}).
\begin{probl}\label{prob: L(2n) lineable} Is $\bigcup_{1\leq n< \omega}\mathrm{L}(2n)$ lineable?
\end{probl}

Similarly, we could ask if $\bigcup_{2\leq n< \omega}\mathrm{L}(n^2)$ is lineable. Note that we don't even know if these sets are $2$-lineable. In connection with Theorem \ref{th: odd c lineable} it is also natural to ask the following.
\begin{probl}\label{prob: L(2n+1) densely} Is $\bigcup_{1\leq n<\omega}\mathrm{L}(2n+1)$ densely lineable in $\ell_\infty$?
\end{probl}

\subsection*{Acknowledgements} We are most grateful to the anonymous referee for a careful reading of the manuscript, for several suggestions which improved the presentation, and for spotting and correcting a gap in the proof of Proposition \ref{prop: comb lin many values}.

\subsection*{Added in proof.} After the completion of our research, we were informed of some new results that were motivated by our paper. In particular, Menet and Papathanasiou \cite{MenetPapa} recently obtained several interesting results that in particular solve Problem \ref{prob: non separable} and Problem \ref{prob: L(2n+1) densely} and imply that $\bigcup_{2\leq n< \omega}\mathrm{L}(n^2)$ is not lineable. While Problem \ref{prob: L(2n) lineable} seems to be still open, Davide Ravasini recently showed that $\bigcup_{1\leq n< \omega}\mathrm{L}(2n)$ is $2$-lineable. We are most grateful to him for allowing us to explain his argument here.

One uses the same notation and construction as in Example \ref{rmk:strangeoctagon}. Let $\mathcal{P}$ be the vertices of a regular $15$-gon and let $\mathcal{T}\subseteq \mathcal{P}$ be the vertices of an equilateral triangle. Then the points $\mathcal{P} \setminus \mathcal{T}$ are as desired. Indeed, if $\mathscr{L}$ is a set of parallel lines with $\mathcal{P}\subseteq \mathscr{L}$ and such that every line of $\mathscr{L}$ contains a point of $\mathcal{P}$, then $|\mathscr{L}|$ equals $15$ or $8$. In the first case, exactly $12$ lines are needed to cover $\mathcal{P} \setminus \mathcal{T}$. So, we can assume that $|\mathscr{L}|=8$. Now, if one edge of $\mathcal{T}$ is parallel to the lines in $\mathscr{L}$, then exactly $2$ of the lines in $\mathscr{L}$ don't contain points of $\mathcal{P} \setminus \mathcal{T}$ (the line containing a unique point of $\mathcal{P}$ actually contains a point of $\mathcal{T}$). In the other case, all $8$ lines contain points of $\mathcal{P} \setminus \mathcal{T}$. Hence, in order to cover $\mathcal{P} \setminus \mathcal{T}$ one needs $6$, $8$, or $12$ parallel lines, which means that $\mathrm{L}(6)\cup \mathrm{L}(8)\cup \mathrm{L}(12)$ is $2$-lineable.

Incidentally, the same construction works for every $(2n-1)$-gon, provided that $n$ is even and $2n-1$ is a multiple of $3$. Hence, for every $k\geq1$, $\mathrm{L}(6k)\cup \mathrm{L}(6k+2)\cup \mathrm{L}(12k)$ is $2$-lineable.

\bibliographystyle{amsplain}
\bibliography{ideale}

\providecommand{\MR}[1]{}
\providecommand{\bysame}{\leavevmode\hbox to3em{\hrulefill}\thinspace}
\providecommand{\MR}{\relax\ifhmode\unskip\space\fi MR }
\providecommand{\MRhref}[2]{%
  \href{http://www.ams.org/mathscinet-getitem?mr=#1}{#2}
}
\providecommand{\href}[2]{#2}
\begin{thebibliography}{10}

\bibitem{MR3445906}
R.~Aron, L.~Bernal-Gonz\'{a}lez, D.M. Pellegrino, and J.B.
  Seoane-Sep\'{u}lveda, \emph{Lineability: the search for linearity in
  mathematics}, Monographs and Research Notes in Mathematics, CRC Press, Boca
  Raton, FL, 2016. \MR{3445906}

\bibitem{AvilTodo}
A.~Avil\'{e}s and S.~Todor\v{c}evi\'c, \emph{Zero subspaces of polynomials on
  {$\ell_1(\Gamma)$}}, J. Math. Anal. Appl. \textbf{350} (2009), no.~2,
  427--435. \MR{2474778}

\bibitem{MR3003676}
A.~Bartoszewicz and S.~G{\l}\c{a}b, \emph{Strong algebrability of sets of
  sequences and functions}, Proc. Amer. Math. Soc. \textbf{141} (2013), no.~3,
  827--835. \MR{3003676}

\bibitem{MR1350295}
T.~Bartoszy\'{n}ski and H.~Judah, \emph{Set theory. {O}n the structure of the
  real line}, A K Peters, Ltd., Wellesley, MA, 1995. \MR{1350295}

\bibitem{BernalCabreraJFA}
L.~Bernal-Gonz\'{a}lez and M.~Ord\'{o}\~{n}ez Cabrera, \emph{Lineability
  criteria, with applications}, J. Funct. Anal. \textbf{266} (2014), no.~6,
  3997--4025. \MR{3165251}

\bibitem{MR3119823}
L.~Bernal-Gonz\'{a}lez, D.M. Pellegrino, and J.B. Seoane-Sep\'{u}lveda,
  \emph{Linear subsets of nonlinear sets in topological vector spaces}, Bull.
  Amer. Math. Soc. (N.S.) \textbf{51} (2014), no.~1, 71--130. \MR{3119823}

\bibitem{Bourbaki}
N.~Bourbaki, \emph{General topology. {C}hapters 1--4}, Springer-Verlag, Berlin,
  1989. \MR{979294}

\bibitem{CarSep14}
D.~Cariello and J.B. Seoane-Sep\'{u}lveda, \emph{Basic sequences and
  spaceability in {$\ell_p$} spaces}, J. Funct. Anal. \textbf{266} (2014),
  no.~6, 3797--3814. \MR{3165243}

\bibitem{Eng}
R.~Engelking, \emph{General topology}, Mathematical Monographs, Vol. 60,
  PWN---Polish Scientific Publishers, Warsaw, 1977. \MR{0500780}

\bibitem{FHHMZ}
M.~Fabian, P.~Habala, P.~H\'{a}jek, V.~Montesinos, and V.~Zizler, \emph{Banach
  space theory. {T}he basis for linear and nonlinear analysis}, CMS Books in
  Mathematics/Ouvrages de Math\'{e}matiques de la SMC, Springer, New York,
  2011. \MR{2766381}

\bibitem{FGK}
V.P. Fonf, V.I. Gurariy, and M.I. Kadets, \emph{An infinite dimensional
  subspace of {$C[0,1]$} consisting of nowhere differentiable functions}, C. R.
  Acad. Bulgare Sci. \textbf{52} (1999), no.~11-12, 13--16. \MR{1738120}

\bibitem{FSTZ}
V.P. Fonf, J.~Somaglia, S.~Troyanski, and C.~Zanco, \emph{Almost overcomplete
  and almost overtotal sequences in {B}anach spaces {II}}, J. Math. Anal. Appl.
  \textbf{434} (2016), no.~1, 84--92. \MR{3404549}

\bibitem{FZ}
V.P. Fonf and C.~Zanco, \emph{Almost overcomplete and almost overtotal
  sequences in {B}anach spaces}, J. Math. Anal. Appl. \textbf{420} (2014),
  no.~1, 94--101. \MR{3229811}

\bibitem{Gurariy}
V.I. Gurariy, \emph{Subspaces and bases in spaces of continuous functions},
  Dokl. Akad. Nauk SSSR \textbf{167} (1966), 971--973. \MR{0199674}

\bibitem{HKRtams}
P.~H\'{a}jek, T.~Kania, and T.~Russo, \emph{Separated sets and {A}uerbach
  systems in {B}anach spaces}, Trans. Amer. Math. Soc. \textbf{373} (2020),
  no.~10, 6961--6998. \MR{4155197}

\bibitem{HRjfa}
P.~H\'{a}jek and T.~Russo, \emph{On densely isomorphic normed spaces}, J.
  Funct. Anal. \textbf{279} (2020), no.~7, 108667. \MR{4107815}

\bibitem{Jech}
T.~Jech, \emph{Set theory. {T}he third millennium edition, revised and
  expanded}, Springer Monographs in Mathematics, Springer-Verlag, Berlin, 2003.
  \MR{1940513}

\bibitem{KitTim}
D.~Kitson and R.M. Timoney, \emph{Operator ranges and spaceability}, J. Math.
  Anal. Appl. \textbf{378} (2011), no.~2, 680--686. \MR{2773276}

\bibitem{Klee}
V.~Klee, \emph{On the {B}orelian and projective types of linear subspaces},
  Math. Scand. \textbf{6} (1958), 189--199. \MR{105005}

\bibitem{MenetPapa}
Q.~Menet and D.~Papathanasiou, \emph{Structure of sets of bounded sequences
  with a prescribed number of accumulation points},
  \href{https://arxiv.org/abs/2303.03871}{\url{ arXiv:2303.03871}}.

\bibitem{Papa2021}
D.~Papathanasiou, \emph{Dense lineability and algebrability of
  $\ell_\infty\setminus c_0$}, Proc. Amer. Math. Soc. \textbf{150} (2022),
  991--996.

\bibitem{MR3698923}
M.J. Perron, \emph{On the {S}tructure of {I}ndependent {F}amilies}, ProQuest
  LLC, Ann Arbor, MI, 2017, Thesis (Ph.D.)--Ohio University. \MR{3698923}

\bibitem{MR2399380}
R.~Pinchasi, \emph{The minimum number of distinct areas of triangles determined
  by a set of {$n$} points in the plane}, SIAM J. Discrete Math. \textbf{22}
  (2008), no.~2, 828--831. \MR{2399380}

\bibitem{PZpoly}
A.~Plichko and A.~Zagorodnyuk, \emph{On automatic continuity and three problems
  of {T}he {S}cottish book concerning the boundedness of polynomial
  functionals}, J. Math. Anal. Appl. \textbf{220} (1998), no.~2, 477--494.
  \MR{1614947}

\bibitem{Rmoutil}
M.~Rmoutil, \emph{Norm-attaining functionals need not contain 2-dimensional
  subspaces}, J. Funct. Anal. \textbf{272} (2017), no.~3, 918--928.
  \MR{3579129}

\bibitem{Talagrand}
M.~Talagrand, \emph{Compacts de fonctions mesurables et filtres non
  mesurables}, Studia Math. \textbf{67} (1980), no.~1, 13--43. \MR{579439}

\bibitem{Wilansky1975}
A.~Wilansky, \emph{Semi-{F}redholm maps of {FK} spaces}, Math. Z. \textbf{144}
  (1975), no.~1, 9--12. \MR{405155}

\end{thebibliography}
\end{document}